\documentclass[a4paper,11pt, dvipsnames]{article} 
\usepackage{fourier} 
\usepackage{amssymb}
\usepackage{amsmath}
\usepackage{amsthm}
\usepackage[margin=1.30in]{geometry}
\usepackage[all]{xy}
\usepackage[english]{babel}
\usepackage{tikz}
\usepackage{tikz-cd}
\usetikzlibrary{matrix}
\usetikzlibrary{calc,intersections}
\usepackage{url}
\usepackage[shortlabels]{enumitem}
\usepackage{multirow}

\usepackage{hyperref}
\usepackage{color}
\usepackage{marvosym} 

\theoremstyle{plain}
\newtheorem{introtheorem}{Theorem}

\newtheorem{introprop}[introtheorem]{Proposition}
\newtheorem{thm}{Theorem}[section]
\newtheorem{prop}[thm]{Proposition}
\newtheorem{lemma}[thm]{Lemma}
\newtheorem{cor}[thm]{Corollary}
\theoremstyle{definition}
\newtheorem{df}[thm]{Definition}
\newtheorem{notn}[thm]{Notation}
\newtheorem*{notn*}{Notation}

\newtheorem{remark}[thm]{Remark}

\newtheorem{example}[thm]{Example}

\DeclareMathOperator{\AV}{AV}
\DeclareMathOperator{\End}{End}
\DeclareMathOperator{\Hom}{Hom}

\DeclareMathOperator{\Tr}{Tr}
\DeclareMathOperator{\Pic}{Pic}
\DeclareMathOperator{\ord}{ord}

\def\Q{\mathbb{Q}}
\def\Z{\mathbb{Z}}
\def\C{\mathbb{C}}
\def\R{\mathbb{R}}
\def\F{\mathbb{F}}
\def\p{\mathfrak{p}}

\newcommand{\cA}{\mathcal{A}}

\newcommand{\cF}{\mathcal{F}}
\newcommand{\cG}{\mathcal{G}}
\newcommand{\cI}{\mathcal{I}}
\newcommand{\cL}{\mathcal{L}}
\newcommand{\cM}{\mathcal{M}}
\newcommand{\cO}{\mathcal{O}}

\newcommand{\cR}{\mathcal{R}}
\newcommand{\cW}{\mathcal{W}}

\newcommand{\into}{\hookrightarrow}
\newcommand{\vphi}{\varphi}
\newcommand{\set}[1]{\left\lbrace#1\right\rbrace }
\newcommand{\AVcs}[1]{\AV^{\text{cs}}({#1})}
\newcommand{\Mcs}[1]{\cM^{\text{cs}}({#1})}
\newcommand{\Palpha}[2]{\mathcal{P}^{\alpha}_{{#1}}({#2})}
\newcommand{\Pone}[2]{\mathcal{P}^{1}_{{#1}}({#2})}

\renewcommand{\bar}{\overline}

\title{Polarizations of abelian varieties over finite fields \\ via canonical liftings}
\author{Jonas Bergstr\"om, Valentijn Karemaker, and Stefano Marseglia}
\date{\today}

\AtEndDocument{\bigskip{\footnotesize\noindent
    \textsc{
    Matematiska institutionen, Stockholms Universitet, SE-106 91 Stockholm, Sweden
    }\\
    \textit{E-mail address}: 
    \text{jonasb@math.su.se}\\
    \\
    \textsc{
    Mathematisch Instituut, Universiteit Utrecht, Postbus 80010, 3508 TA Utrecht,
    The Netherlands
    }\\
    \textit{E-mail address}: 
    \text{v.z.karemaker@uu.nl}\\
    \\
    \textsc{
    Mathematisch Instituut, Universiteit Utrecht, Postbus 80010, 3508 TA Utrecht,
    The Netherlands
    }\\
    \textit{E-mail address}: 
    \text{s.marseglia@uu.nl}
}}

\begin{document}
\maketitle

\begin{abstract}
We describe all polarizations for all abelian varieties over a finite field in a fixed isogeny class corresponding to a squarefree Weil polynomial, when one variety in the isogeny class admits a canonical lifting to characteristic zero, i.e., a lifting for which the reduction morphism induces an isomorphism of endomorphism rings. 
Categorical equivalences between abelian varieties over finite fields and fractional ideals in {\'e}tale algebras enable us to explicitly compute isomorphism classes of polarized abelian varieties satisfying some mild conditions. We also implement algorithms to perform these computations.
\end{abstract}

\section{Introduction}

Polarizations, and in particular principal polarizations, play a key role in the study of abelian varieties. 
Elliptic curves are for instance canonically principally polarized. Among higher-dimensional varieties, the same holds true for Jacobians of curves, which are extensively studied. But while any abelian variety admits a polarization of some degree, the existence and abundance of polarizations of a prescribed degree are less clear.

The extra data provided by polarizations introduce a geometric rigidification which is necessary to define and study the moduli space $\mathcal{A}_{g,d}$ of $g$-dimensional abelian varieties with a polarization of degree $d^2$. 

Over any field, polarizations can be defined as isogenies from the abelian variety to its dual that are induced from ample line bundles. Although very precise, this definition in general does not translate into an explicit description.
Over the complex numbers however, and by extension over any field of characteristic zero, the definition can be made explicit in terms of Riemann forms, which is very convenient for performing computations. Such a description is currently lacking in positive characteristic.

As a result, there are many open problems about polarized abelian varieties in characteristic~$p$,
such as determining which isogeny classes over finite fields that are principally polarizable, and also about their moduli spaces $\mathcal{A}_{g,d}\otimes \F_p$, such as determining the Hecke orbits, maximal subvarieties, and Newton polygon strata.\\

In this paper we aim to \emph{determine} and \emph{compute} polarizations for abelian varieties over finite fields. 
When such a variety lifts to characteristic zero, in such a way that all of its endomorphisms also lift, we again have the full description of polarizations in characteristic zero at our disposal. This type of lifting is called a canonical lifting.

This idea has been exploited by Howe \cite{Howe95}, building on work by Deligne \cite{Del69}, to describe polarizations for \emph{ordinary} abelian varieties; that is, abelian varieties~$A$ with $p$-rank equal to $\dim(A)$. Namely, it follows from Serre-Tate theory that every ordinary abelian variety over a finite field has a canonical lifting. Under some assumptions, the same holds for \emph{almost-ordinary} abelian varieties~$A$, that is, those whose $p$-rank equals  $\dim(A)-1$,  see Oswal-Shankar \cite{OswalShankar19}. We generalize their results in~Proposition~\ref{prop:almord_lift} and Theorem~\ref{thm:OS}. Using these results, we compute isomorphism classes of principally polarized almost-ordinary abelian varieties in Example
\ref{ex:nonord2}.

It is known that abelian varieties of lower $p$-rank generally do not have canonical liftings. Our main theorem (Theorem~\ref{thm:main1}, paraphrased as Theorem~\ref{thm:introA} below) nevertheless gives a complete description of all polarizations (up to polarized isomorphisms) of abelian varieties over a finite field in a fixed isogeny class, as long as there exists at least one variety in this isogeny class which admits a canonical lifting. 
This description crucially relies on an identification between the polarizations in characteristics zero and $p$, which is expressed in terms of a totally real unit element $\alpha$ of the endomorphism ring of the variety.

In order to effectively and algorithmically compute the polarizations, we further need a functorial description of the fixed isogeny class in terms of computable algebraic objects. 
The results of this paper are built on the equivalence of categories provided by Centeleghe-Stix~\cite{CentelegheStix15}. This equivalence is used in Theorem~\ref{thm:introA}, but notably also in Proposition~\ref{prop:comppol}, paraphrased here as Proposition~\ref{prop:introB}.

\begin{introtheorem}\label{thm:introA}
Fix an isogeny class over $\mathbb{F}_p$ corresponding to a squarefree Weil polynomial $h$ and a CM-type $\Phi$ for the CM-algebra $L = \mathbb{Q}[x]/(h)$. Suppose that the isogeny class contains an abelian variety $A_0$, which admits a canonical lifting with CM by $L$ through $\Phi$ over a $p$-adic field, and whose endomorphism ring $S = \End(A_0)$ is Gorenstein and stable under the canonical involution of $L$, i.e., $S = \overline{S}$. There exists a totally real unit element $\alpha \in S^*$ with the following property: 
If $B_0$ is any other abelian variety in the isogeny class with dual $B_0^{\vee}$, let $\mu: B_0 \to B_0^{\vee}$ be any isogeny and $\lambda\in L$ its image under the Centeleghe-Stix equivalence.
Then $\mu$ is a polarization if and only if $\gamma = \alpha^{-1} \lambda$ satisfies $\overline{\gamma} = -\gamma$ (``totally imaginary'') and $\Im(\varphi(\gamma)) > 0$ for every $\varphi \in \Phi$ (``$\Phi$-positive'').
\end{introtheorem}

We note here that we need the assumptions on the endomorphism ring $S$ to ensure that~$\alpha$ is not just any invertible element in $L$ but that it is contained in $S^*$, which is important for the following proposition and theorems. 

\begin{introprop}\label{prop:introB}
We use the notation and assumptions of Theorem~\ref{thm:introA}. To any endomorphism ring $T$ of an abelian variety isogenous to $A_0$ we can associate a finite set~$\mathcal{T} \subseteq L$. Suppose that we have an abelian variety~$B_0$ isogenous to $A_0$, with endomorphism ring $T$. If $B_0$ is isomorphic to its dual, then there exists an element $i_0 \in L^*$ such that the image under the Centeleghe-Stix equivalence of a complete set of representatives of principal polarizations of~$B_0$ up to isomorphism is the set 
\begin{equation}\label{eq:introprobB}
    \{ i_0  u: u \in \mathcal{T} \text{ such that } \alpha^{-1} i_0 u \text{ is totally imaginary and $\Phi$-positive } \}.
\end{equation}
Moreover, the element $i_0$ and finite set $\mathcal{T}$ can be explicitly computed.
\end{introprop}

Using the Centeleghe-Stix equivalence, combined with a construction due to the third author \cite{MarAbVar18}, we show how to describe the abelian varieties (and their polarizations) in a \emph{squarefree} isogeny class (i.e., abelian varieties whose characteristic polynomial of Frobenius has no repeated factors) in terms of fractional ideals in an {\'e}tale algebra over~$\Q$.
This functorial description has the benefit of allowing many more isogeny classes than those considered in~\cite{Del69}, \cite{Howe95} and~\cite{OswalShankar19}, but has the drawback that it only applies to varieties defined over the prime field $\F_p$. In future work, we plan to investigate how to extend our results to abelian varieties defined over more general finite fields. And while we restrict ourselves to squarefree isogeny classes in the current paper, mostly to simplify the computations, we also aim to study how to generalize the results to other isogeny classes.

Via complex uniformization, abelian varieties (and their polarizations) in characteristic zero can also be described by fractional ideals; comparing these with the ideals obtained through the categorical equivalence is done using the element $\alpha$ (in Theorem~\ref{thm:introA}) which can pose an obstruction to explicit computations. Notably, Example~\ref{ex:x8+16} shows that such an obstruction actually occurs. In Subsection~\ref{ssec:polcomp} we describe exactly when $\alpha = 1$ or when we can compute polarizations regardless of its value. 
The main result in Subsection~\ref{ssec:polcomp} is Theorem \ref{thm:eff}, stated here in a less general form:
\begin{introtheorem}\label{eq:introthmC}
We use the notation and assumptions of Proposition~\ref{prop:introB}.
Denote by $S^*_\R$ (resp.~$T^*_\R$) the group of totally real units of $S$ (resp.~$T$).
If $S^*_\R\subseteq T^*_\R$, then the set in Equation~\eqref{eq:introprobB} is in bijection with the set
\begin{equation*}
    \{ i_0 u: u \in \mathcal{T} \text{ such that } i_0 u \text{ is totally imaginary and $\Phi$-positive } \},
\end{equation*}
which does not depend on $\alpha$.
\end{introtheorem}

Sometimes, even when the hypotheses of Theorem \ref{eq:introthmC} (or Theorem \ref{thm:eff}) do not hold, we can still compute the number of isomorphism classes of principally polarized abelian varieties with a given endomorphism ring; for instance, this happens
when the numerical outcomes from all different CM-types agree.
We refer the reader to Theorem \ref{thm:sameoutput}, which is simplified here as Theorem \ref{eq:introthmD}:
\begin{introtheorem}\label{eq:introthmD}
We use the notation and assumptions of Proposition~\ref{prop:introB}.
Assume that there are~$r$ isomorphism classes of abelian varieties 
isogenous to $A_0$ with endomorphism ring $T$, represented by abelian varieties $B_1, \ldots, B_r$. If each $B_k$, $k=1, \ldots, r$, is isomorphic to its dual, there exist elements $i_k \in L^*$ as in Proposition~\ref{prop:introB}. 
For any CM-type $\Phi'$, we put
\[   \Pone{\Phi'}{B_k}=\{ i_k  u: u \in \mathcal{T} \text{ such that } i_k u \text{ is totally imaginary and $\Phi'$-positive } \}. \]
If there exists a non-negative integer $N$ such that for every CM-type $\Phi'$ we have
\[
    \vert \Pone{\Phi'}{B_1} \vert + \ldots + \vert \Pone{\Phi'}{B_r} \vert = N,
\]
then there are exactly $N$ isomorphism classes of principally polarized abelian varieties with endomorphism ring $T$. 
\end{introtheorem}

In Example \ref{ex:detailedlowprank} we describe in detail how to apply Theorems \ref{eq:introthmC} and \ref{eq:introthmD}. 

While in this paper we study the application of the categorical equivalence due to Centeleghe and Stix to squarefree abelian varieties, we believe our results can be extended to apply to more general abelian varieties over more general finite fields, as alluded to above. For instance, we could adapt our techniques to other functorial equivalences, such as those given by Serre~\cite[Appendix]{Lauter01_Serre}, Kani \cite{Kani11}, Amir-Khosravi~\cite{khosravi}, and Jordan et al.~\cite{JKPRSBT18}, which all focus on abelian varieties which are isogenous to a power. In fact, for powers of ordinary elliptic curves, an algorithm to compute polarizations based on the equivalence in~\cite{JKPRSBT18} was recently given in~\cite{KNRR} by Kirschmer et al.

This paper fits into a larger programme to systematically study isogeny classes of abelian varieties (with or without polarizations) over finite fields. For ordinary principally polarized abelian varieties, these studies have also been carried out in terms of product formulas and orbital integrals, cf.~\cite{AchterGordon17} for elliptic curves and \cite{AchterWilliams15} for abelian surfaces. For supersingular abelian varieties, lattices in (free) modules over quaternion algebras have been employed, e.g. in \cite{XYY, XYYII, XYYIII} for superspecial abelian surfaces. The special case of studying the existence of Jacobian varieties in a fixed isogeny class has been addressed, using a variety of tools, in e.g.~\cite{MaNart02, how04, HNR08, HNR09, valentijnetal}.\\

The paper is structured as follows. Since all abelian varieties over finite fields have complex multiplication (CM), in Section \ref{sec:lifting} we first collect some useful results on CM abelian varieties, namely on homomorphisms (in Subsection~\ref{ssec:homs}) and on canonical liftings (in Subsection~\ref{ssec:lift}). The latter include the aforementioned lifting results for ordinary and almost-ordinary abelian varieties (Propositions~\ref{prop:ord_lift} and \ref{prop:almord_lift}), the residual reflex condition (Definition~\ref{def:RRC}) -- which, by a result of Chai, Conrad and Oort~\cite{chaiconradoort14}, is a sufficient condition for a canonical lifting to exist (cf.~Theorem~\ref{thm:CCOlift}) -- and a result (Proposition~\ref{prop:sepisoglifting}) stating that abelian varieties which are separably isogenous to a liftable variety are liftable.

In Section~\ref{sec:cats} we introduce the algebraic and functorial descriptions of abelian varieties that facilitate the computation of polarizations; Subsection~\ref{ssec:unif} contains the description for varieties in characteristic zero, based on complex uniformization, while Subsections~\ref{ssec:CS} and~\ref{ssec:Stefano} contain that for varieties in positive characteristic. In particular, in Proposition~\ref{prop:CSdual} we prove how the functor behaves with respect to duality. Finally, Subsection~\ref{ssec:compred} compares the compatibility of these descriptions under reduction of the varieties considered. 

Section~\ref{sec:pols} contains the main result (Theorem~\ref{thm:main1}) characterizing polarizations. Beforehand, in Subsection~\ref{ssec:spreading}, we discuss how polarizations of a given abelian variety induce polarizations on all abelian varieties in its isogeny class.

Section~\ref{sec:algo} contains our computational results. In Subsection~\ref{ssec:RRCcomp} we implement a verification of the residual reflex condition, and in Subsection~\ref{ssec:polcomp} we describe an algorithm to find a complete set of representatives of the principal polarizations of an abelian variety in a fixed isogeny class (Proposition~\ref{prop:comppol}). Finally, Subsection~\ref{ssec:examples} contains a number of explicit examples to illustrate the scope and strength of our algorithms.

The code to compute polarizations in a squarefree isogeny class is available at \url{https://github.com/stmar89/PolsAbVarFpCanLift}.
At the same link one can also find the code to compute isomorphism classes and polarizations for squarefree almost-ordinary isogeny classes over an arbitrary finite field $\F_q$.
Moreover, at \href{https://github.com/stmar89/PolsAbVarFpCanLift/blob/main/additional_examples.pdf}{{\tt additional\_examples.pdf}} on GitHub we have summarized our computations of the principal polarizations for all squarefree isogeny classes of abelian varieties of dimension two and three over the finite fields $\F_3$, $\F_5$ and $\F_7$, and of dimension four over $\F_2$ and $\F_3$.\\

\noindent {\bfseries Notational conventions.} When considering abelian varieties $A, A'$ over a field~$K$, we drop the suffix $K$ and write $\End(A)$, $\End^0(A) := \End(A) \otimes \Q$, $\Hom(A,A')$, and $\Hom^0(A,A') := \Hom(A,A') \otimes \Q$ instead of $\End_K(A)$, $\End^0_K(A)$, $\Hom_K(A,A')$ and $\Hom^0_K(A,A')$, respectively.

\section*{Acknowledgments}
The authors thank Knut och Alice Wallenbergs Stiftelse for financial support through grants MG2018-0044 and 2017.0418. The second author is supported by the Dutch Research Council (NWO) through grant VI.Veni.192.038. The third author thanks the Max Planck Institute for Mathematics, and NWO grants 613.001.651 and VI.Veni.202.107 for support. The third author would like to express his gratitude to Frans Oort for helpful discussions and to the Simons Collaboration on Arithmetic Geometry, Number Theory, and Computation for granting access to their computational facilities. The authors thank Frans Oort, Christophe Ritzenthaler, Andrew Sutherland, and John Voight for helpful comments on an earlier draft, and the anonymous referee for careful reading and useful suggestions.

\section{Abelian varieties with complex multiplication}\label{sec:lifting}

In this section, we consider abelian varieties with complex multiplication, either defined over fields of characteristic zero or over finite fields $\mathbb{F}_q$ of characteristic~$p$. 
In Subsection~\ref{ssec:homs}, we study endomorphisms of and homomorphisms between such varieties. Secondly, in Subsection~\ref{ssec:lift}, we study liftings of abelian varieties over finite fields and state results that guarantee the existence of such liftings.

\subsection{Homomorphisms of CM abelian varieties}\label{ssec:homs}

\begin{df}\label{def:CM}\
\begin{enumerate}
    \item A number field $L/\Q$ is a \emph{CM-field} if it is totally complex and quadratic over a totally real subfield. It is equipped with a canonical involution $x \mapsto~\overline{x}$.
    A \emph{CM-algebra} is a finite product of CM-fields. 
    \item An abelian variety $A$ over a field $K$ is \emph{of CM-type} (or admits \emph{sufficiently many complex multiplications}) if there exists a commutative semisimple $\Q$-subalgebra $L$ in $\mathrm{End}^0(A)$ with rank $2g$ over $\Q$, which can be taken to be a CM-algebra. In the situation where $L$ is given, we will also say that $A$ has \emph{CM by $L$}. 
    As an equivalent definition, each of the simple factors $B_i$ of $A$ satisfies that its endomorphism algebra $\mathrm{End}^0(B_i)$ contains a field $L_i$ of degree $[L_i : \Q] = 2\mathrm{dim}(B_i)$.
    It is known that every abelian variety over a finite field is of CM-type. 
\end{enumerate}
\end{df}

\begin{notn}\label{not} Let $A$ and $A'$ be abelian varieties defined over a field $K$ and with CM by an algebra $L$. Fix embeddings $i:L\into \End^0(A)$ and $i':L\into \End^0(A')$.
We denote by $\Hom_L(A,A')$ the group of \emph{$L$-linear homomorphisms} with respect to the fixed embeddings, that is, the homomorphisms $f~:~A \to A'$ such that for every $\alpha \in i^{-1}(\End(A)) \cap i'^{-1}(\End(A'))$ we have
\[ 
i'(\alpha)\circ f = f \circ i(\alpha). 
\]
We will also write $\Hom^0_L(A,A') = \Hom_L(A,A') \otimes \Q$, $\End_L(A)=\Hom_L(A,A)$, and $\End^0_L(A) = \End_L(A) \otimes \Q$.
In what follows, an \emph{$L$-linear map} (between groups of $L$-linear homomorphisms) will mean a map which is compatible with the $L$-linear structures.
\end{notn}

\begin{lemma}\label{lem:KtoKbar}
Let $A$ and $A'$ be abelian varieties defined over a field $K$. Suppose that $A$ and $A'$ have CM by an algebra $L$.
Then base change from $K$ to its algebraic closure $\overline{K}$ induces an $L$-linear bijection
    \[
    \mathrm{Hom}_L(A,A') \leftrightarrow \mathrm{Hom}_L(A_{\overline{K}},A'_{\overline{K}}).
    \]
This bijection also holds when $\mathrm{Hom}$ is replaced with $\mathrm{Hom}^0$.
\end{lemma}
\begin{proof}
See
\cite[Example 1.7.4.1]{chaiconradoort14}, which works without any assumptions on the endomorphism rings of $A$ and $A'$.
\end{proof}

\begin{lemma}\label{lem:KbartoC}
Let $K$ be a field with algebraic closure $\overline{K}$ such that there exists an embedding $j: \overline{K} \into \mathbb{C}$, and let $A$ and $A'$ be abelian varieties defined over $K$.
Then the base change induces a bijection
    \[
    \mathrm{Hom}(A_{\overline{K}},A'_{\overline{K}}) \leftrightarrow \mathrm{Hom}(A_{\mathbb{C}},A'_{\mathbb{C}}).
    \]
This bijection also holds when $\mathrm{Hom}$ is replaced with $\mathrm{Hom}^0$.
\end{lemma}
\begin{proof}
 See \cite[Lemma~1.2.1.2]{chaiconradoort14}.
\end{proof}

Lemmas~\ref{lem:KtoKbar} and \ref{lem:KbartoC} immediately imply the following corollary.

\begin{cor}\label{cor:allhoms}
Let $K$ be a number field and $K_{\mathfrak{p}}$ be its completion at a prime $\mathfrak{p}$.
Suppose that $A$ and $A'$ are defined over $K$ and have CM by an algebra $L$. Then there are $L$-linear bijections
\begin{equation*}
\begin{split}
\Hom_L(A,A') & \leftrightarrow \Hom_L(A_{\overline{K}}, A'_{\overline{K}}) \leftrightarrow \Hom_L(A_{\C},A'_{\C}) \\ & \leftrightarrow  \Hom_L(A_{\overline{K}_{\mathfrak{p}}},A'_{\overline{K}_{\mathfrak{p}}}) \leftrightarrow \Hom_L(A_{K_{\mathfrak{p}}}, A'_{K_\mathfrak{p}}).
\end{split}
\end{equation*} 
These bijections also hold when $\mathrm{Hom}$ is replaced with $\mathrm{Hom}^0$.
\qed
\end{cor}

\subsection{CM-liftings}\label{ssec:lift}

Throughout this subsection, we let $A_0/\F_q$ be a $g$-dimensional abelian variety whose Frobenius endomorphism $\pi_{A_0}$ has characteristic polynomial~$h$. 
If $h$ is squarefree, that is, its irreducible factors $h=  h_1 \cdot \ldots \cdot h_t$ are all distinct, then $L = \Q[x]/(h) \simeq \Q(\pi_1) \times \ldots \times \Q(\pi_t) =: L_1 \times \ldots \times L_t$ is a CM-algebra. If $h$ is irreducible, then $A_0$ is \emph{simple} (over $\F_q$), meaning that it has no nonzero proper abelian subvariety. Some of the results we recall in this subsection were originally, and more generally, stated for \emph{isotypic} abelian varieties, for which $h$ has a unique irreducible factor. 

\begin{df}\label{def:CMlift}
Let $A_0/\F_q$ be a $g$-dimensional abelian variety.
We say that $A_0/\F_q$ has a \emph{CM-lifting}
if there exist both a local domain $\mathcal{R}$ of characteristic zero with residue field $\F_q$ and fraction field $K$, and a $g$-dimensional abelian scheme $\mathcal{A}$ over $\mathcal{R}$, such that $\mathcal{A} \otimes \F_q \simeq A_0$ and $\mathcal{A} \otimes K =: A$, where $\mathrm{End}^0(A)$ 
contains a CM-algebra $L$ of degree $[L:\Q] = 2g$. We also say that $A/K$ (or $\mathcal{A}/\mathcal{R}$) is a CM-lifting of $A_0/\F_q$.
The scheme $\mathcal{A}/\mathcal{R}$ is a \emph{canonical lifting} of $A_0/\F_q$ if $\mathrm{End}(\mathcal{A}) = \mathrm{End}(A_0)$.
\end{df}

To recall some first results from the literature on the existence of canonical liftings, we need the following definition.

\begin{df}\label{def:ordAV} \
\begin{enumerate}
    \item The \emph{$p$-rank} $f(A_0)$ of $A_0/\F_q$ is defined through
    \[
    \vert A_0(\overline{\F}_p)[p]\vert = p^{f(A_0)},
    \]
    and $0 \leq f(A_0) \leq g$. The $p$-rank is an isogeny invariant.
    \item $A_0$ is \emph{ordinary} if $f(A_0) = g$ and \emph{almost ordinary} if $f(A_0) = g-1$.
\end{enumerate}
\end{df}

\begin{prop}[{\cite[Theorem 2.1]{katz}, \cite[Theorem V.3.3]{messing}}]
\label{prop:ord_lift} 
Let $A_0$ be an ordinary abelian variety over $\F_q$. Then $A_0$ admits a canonical lifting to the Witt vectors~$W(\F_q)$.
\end{prop}

\begin{remark}\label{rem:ordlift}
The canonical lifting of Proposition~\ref{prop:ord_lift} is called a \emph{Serre-Tate lifting}. The Serre-Tate liftings (cf.~\cite[Th{\'e}or{\`e}me 7]{Del69}) provide an equivalence of categories sending an ordinary abelian variety over $\mathbb{F}_q$ to its associated \emph{Deligne module}, i.e., a free finitely generated $\mathbb{Z}$-module equipped with a semisimple endomorphism which plays the role of the Frobenius endomorphism. 

The Centeleghe-Stix equivalence, which we will introduce in Subsection~\ref{ssec:CS} and will use throughout this paper, does not make use of liftings. 
Nevertheless, on the subcategory of ordinary abelian varieties over $\mathbb{F}_p$ this equivalence is isomorphic to Deligne's equivalence, cf.~\cite[7.4]{CentelegheStix15}. 
\end{remark}

\begin{prop}\label{prop:almord_lift}
Let $A_0$ be an almost-ordinary abelian variety over a finite field $\F_q$ with commutative endomorphism ring. Then $A_0$ admits a canonical lifting to a quadratic ramified extension of the Witt vectors $W(\F_q)$. If $q$ is odd, then this extension can be taken to be generated by $\sqrt p$. 
\end{prop}

\begin{proof} For the first statement, see \cite[Remark 3.5.7]{chaiconradoort14} and the references therein.  
For the second statement, see \cite[Section 2]{OswalShankar19}. The result in the latter reference is shown for simple abelian varieties of dimension greater than two. On the other hand, those assumptions are only used to ensure that the endomorphism ring of $A_0$ is commutative, and hence the endomorphism ring of its $p$-divisible group has a supersingular part of rank two. This argument also holds when $W(\overline{\F}_q)$ is replaced with $W(\F_q)$; alternatively, one may use a descent argument following, e.g., \cite[Theorem~1.7.2.1 (Shimura-Taniyama)]{chaiconradoort14}.
\end{proof}

\begin{remark}\label{rem:almord}
As in the ordinary case, cf.~Remark~\ref{rem:ordlift}, the canonical liftings of Proposition~\ref{prop:almord_lift} provide
an equivalence of categories which associates with any almost-ordinary abelian variety over $\mathbb{F}_q$ an \emph{almost-ordinary Deligne module}, cf.~\cite[Definition 3.1]{OswalShankar19}. These modules again have a semisimple endomorphism playing the role of the Frobenius endomorphism; in particular, the characteristic polynomial of the endomorphism is equal to the characteristic polynomial of the Frobenius endomorphism of the abelian variety.

More precisely, in \cite{OswalShankar19} the following is shown for a simple almost-ordinary abelian variety $A_0$ over $\mathbb{F}_q$ whose Frobenius endomorphism has characteristic polynomial~$h$:
If the quadratic extension $L_{ss}/ \mathbb{Q}_p$ corresponding to the slope $1/2$ part of the variety is ramified (resp.~inert), the map from almost-ordinary abelian varieties isogenous to $A_0$ to almost-ordinary Deligne modules with characteristic polynomial $h$ is bijective (resp.~two-to-one), cf.~\cite[Proposition 3.8]{OswalShankar19} (resp.~\cite[Proposition 3.3]{OswalShankar19}).
\end{remark}

We generalize some results in \cite{OswalShankar19} -- in particular Propositions~3.3 and~3.8 mentioned in Remark~\ref{rem:almord}, and Theorem~4.5 which concerns polarizations -- as follows:

\begin{thm} \label{thm:OS} 
Let $q$ be any odd prime power and consider any almost-ordinary isogeny class corresponding to a squarefree characteristic polynomial of Frobenius. 
Then the straightforward generalizations of Proposition 3.3, Proposition 3.8 and Theorem 4.5 in \cite{OswalShankar19} hold.  
\end{thm}

\begin{proof}
The statements in \cite{OswalShankar19} are shown for simple abelian varieties of dimension greater than two. Just as in the proof of  Proposition~\ref{prop:almord_lift}, the arguments in \cite{OswalShankar19}  generalize as long as the endomorphism rings of the abelian varietes are commutative, which in turn holds if and only if the characteristic polynomial is squarefree.
\end{proof}

For an abelian variety $A_0/\F_q$ which is not ordinary nor almost ordinary, a CM-lifting need not exist. 
For instance, \cite[Theorem B]{OortCMliftings} shows that for any prime $p$ and integers $f,g$ such that $0 \leq f \leq g-2$, there exists an abelian variety $B$ over $\overline{\F}_p$ with $\mathrm{dim}(B) = g$ and $f(B) = f$ which does not have a CM-lifting.
In particular, this implies that even after a field extension of $\F_q$, there are varieties which do not have a CM-lifting.

On the other hand, any abelian variety $A_0/\F_q$ is $\F_q$-isogenous to an abelian variety
which \emph{does} admit a CM-lifting:
The classical results by Honda and Tate 
(cf.\cite[Th{\'e}or{\`e}me 2]{tateclasses} and \cite[\textsection 2, Theorem 1]{Honda68}, also see \cite{ChaiOort} for a proof avoiding complex uniformization), may require a finite extension of $\F_q$ for this, while \cite[Theorem 4.1.1]{chaiconradoort14} does not.

However, we will insist on lifting $A_0/\F_q$ to a scheme $\mathcal{A}$ over a \emph{normal} domain~$\mathcal{R}$. 
Only in this case, constructing the generic fiber $A_K$ of $\mathcal{A}$ is a fully faithful functor, cf.~\cite[Lemma 1.8.4]{chaiconradoort14}. 
This implies we can identify $\End(A_K) = \End(\mathcal{A})$ and $\Hom(A_K,A_K^{\vee}) = \Hom(\mathcal{A},\mathcal{A}^{\vee})$, which will be crucial when we consider polarizations in Section~\ref{sec:pols}.

There are obstructions to the existence of a CM-lifting to a normal domain up to isogeny. In \cite[Theorem 2.1.6]{chaiconradoort14}, recalled in Theorem \ref{thm:CCOlift} below, it is proven that such CM-liftings exist if some arithmetic conditions, recalled in Definition~\ref{def:RRC}, are satisfied. 
To state the conditions, we first need another definition.

\begin{df}\label{def:padicCMtype}\
\begin{enumerate}[1.]
    \item \label{def:padicCMtype:cmtype}  A \emph{CM-type} for $L = L_1 \times \ldots \times L_t$ is a subset $\Phi = \Phi_1 \times \ldots \times \Phi_t$ of the embeddings $\Hom(L,\overline{\Q}) = \bigsqcup_{i=1}^t \Hom(L_i, \overline{\Q})$
    of cardinality $g$ such that $\Phi \sqcup (\Phi \circ \iota) = \Hom(L,\overline{\Q})$, where $\iota$ denotes complex conjugation.
    Via the inclusions $\overline{\Q} \subseteq \overline{\Q}_p \subseteq \C$, we may analogously consider $\overline{\Q}_p$-valued (or $p$-adic) and complex-valued CM-types for~$L$. 
    \item \label{def:padicCMtype:through}
    For an abelian variety $A/\overline{\Q}$ with CM by $L$ we have an induced action of $L$ on $\mathrm{Lie}(A)$. This action has $g = \dim(A)$ one-dimensional eigenspaces on which $L$ acts via some embedding $\varphi \in \Hom(L,\overline{\Q})$. These $g$ embeddings make up a CM-type $\Phi$, see \cite[Lemma 1.5.2]{chaiconradoort14}, and we will say that $A$ has CM by $L$ \emph{through} $\Phi$. We may analogously consider $\overline{\Q}_p$-valued or $\C$-valued CM-types. 
    \item \label{def:padicCMtype:reflexfield}
    For a $K$-valued CM-type of $L$, with $K=\bar\Q, \bar\Q_p$, or $\C$, the \emph{reflex field} $E=E_{(L,\Phi)} \subseteq \overline{\Q}\subseteq K$ of $(L,\Phi)$ is defined as the number field such that $\mathrm{Gal}(K/E)$ stabilizes $\Phi$.
    Equivalently, we have 
    \[ 
    E=\Q\left( \sum_{\vphi\in \Phi} \vphi(\alpha) : \alpha \in L \right).
    \]
    We see that $E$ is the compositum $E_1 \cdot \ldots \cdot E_t$, where
    \[
    E_i=\Q\left( \sum_{\vphi\in \Phi_i} \vphi(\alpha) : \alpha \in L_i \right)
    \]
    for $i \in \{1, \ldots, t\}$ is the reflex field of the induced CM-type $\Phi_i$ of $L_i$.
\end{enumerate}
\end{df}

The following definition and theorem are generalizations to squarefree characteristic polynomials of their analogues for irreducible characteristic polynomials found in \cite{chaiconradoort14}. Definition~\ref{def:RRC} will be implemented in Section~\ref{sec:algo}.

\begin{df}[cf.~{\cite[Section 2.1.5]{chaiconradoort14}}]\label{def:RRC}
Let $A_0/\F_q$ be a $g$-dimensional abelian variety whose Frobenius endomorphism $\pi_{A_0}$ has squarefree characteristic polynomial~$h =  h_1 \cdot \ldots \cdot h_t$, with CM by the algebra $L = \Q[x]/(h) \simeq \Q(\pi_1) \times \ldots \times \Q(\pi_t) =: L_1 \times \ldots \times L_t$, and write $\pi_{A_0} = \pi = (\pi_1, \ldots, \pi_t)$.
Let $\Phi = \Phi_1 \times \ldots \times \Phi_t$ be a $p$-adic CM-type for $L$ and let $\nu$ be any place of $L$ above~$p$; then $\nu$ is nontrivial on exactly one factor $L_j$ of $L$. In particular, the completion $L_{\nu} = (L_j)_{\nu}$ and the valuation $\ord_\nu(\pi) = \ord_\nu(\pi_j)$ are well-defined. 

The pair $(L,\Phi)$ satisfies the \textit{generalized residual reflex condition} with respect to~$\pi$ if the following conditions are met:
\begin{enumerate}[1.]
\item \label{def:RRC_item_st} The \emph{Shimura-Taniyama formula} holds for $\pi$. That is, for every place $\nu$ of $L$ above $p$,
\begin{equation*}
\dfrac{\ord_\nu(\pi)}{\ord_\nu(q)}=\dfrac{\#\set{ \vphi \in \Phi \text{ s.t.~} \vphi \text{ induces } \nu }}{[L_\nu:\Q_p]}.
\end{equation*}
\item \label{def:RRC_item_refl} Let $E\subseteq\bar\Q \subseteq \bar\Q_p$ be the 
reflex field attached to $(L,\Phi)$, and let $\nu_p$ be the induced $p$-adic place of~$E$. Then the residue field $k_{\nu_p}$ of the ring of integers of $E_{\nu_p}$ can be realized as a subfield of~$\F_q$.
\end{enumerate}
\end{df}

\begin{remark}\label{rem:genRRConsimplefactors}
In Definition~\ref{def:RRC}, it is possible for each of the pairs $(L_i, \Phi_i)$ to satisfy the generalized residual reflex condition with respect to $\pi_i$, while $(L,\Phi)$ does not satisfy it with respect to $\pi$; see Example~\ref{ex:noCCOprod}.
\end{remark}

\begin{thm}[cf.~{\cite[Theorem 2.1.6, Theorem 2.5.3]{chaiconradoort14}}] \label{thm:CCOlift}
Let $A_0/\F_q$ be an abelian variety such that the characteristic polynomial $h$ of its Frobenius endomorphism is squarefree. Let  $L = \Q[x]/(h)$, and let $\Phi$ be a $p$-adic CM-type for $L$ with associated reflex field~$E$.
Suppose that $(L,\Phi)$ satisfies the generalized residual reflex field condition of Definition~\ref{def:RRC} with respect to $\pi_{A_0}$ .
Then there exists a finite extension $E \subseteq E' \subseteq \overline{\Q}_p$, with induced $p$-adic place $\nu'$ and residue field $\kappa_{\nu'} \simeq \F_q$, and a $g$-dimensional abelian variety $A/E'$ with good reduction at $\nu'$ and with CM by $L$ through $\Phi$, such that there exists an $L$-linear isogeny
\[
A \otimes \kappa_{\nu'} \sim_L A_0.
\]
\end{thm}

\begin{remark}\label{rem:OL}
By \cite[Proposition 1.7.4.5]{chaiconradoort14}, any abelian variety $A$ defined over a field $K$, with CM by an algebra $L$ and $i_A: L \hookrightarrow \mathrm{End}^0(A)$, is $L$-linearly isogenous to an abelian variety $A'/K$ such that $\mathcal{O}_L \subseteq i_{A'}^{-1}(\mathrm{End}(A'))$, where $\mathcal{O}_L$ is the ring of integers of $L$; for ease of notation, for the rest of the remark we suppress the inclusion maps.

In particular, an abelian variety $A_0/\F_q$ satisfying the hypotheses of Theorem~\ref{thm:CCOlift} is $L$-linearly isogenous to a $A'_0/\F_q$ with $\mathrm{End}(A'_0) = \mathcal{O}_L$.
Suppose that $A'_0$ lifts to an abelian scheme~$\mathcal{A}'$ over a normal
domain $\mathcal{R}$ (e.g. a Dedekind domain) with fraction field $K$ and let $A' = \mathcal{A}' \otimes K$.
Then $A'$ may not satisfy that $\mathcal{O}_L \subseteq \mathrm{End}(A')$; however, the same result implies that this holds after another $L$-linear isogeny $A' \sim_L A''$. (Note on the other hand that if $\mathcal{O}_L \subseteq \mathrm{End}(C)$ for some $C/K$ with good reduction, then its reduction $\overline{C}$ will satisfy $\mathcal{O}_L \subseteq \mathrm{End}(\overline{C})$.)
By $\mathcal{R}$-flatness, $\mathcal{O}_L \subseteq \mathrm{End}(A')$ implies that $\mathcal{O}_L \subseteq \mathrm{End}(\mathcal{A}')$, cf.~\cite[2.1.4.1--2.1.4.3]{chaiconradoort14}. 

This property plays a key role in the proof of Theorem~\ref{thm:CCOlift}, and allows us to rephrase its statement in the following form. 
\end{remark}

\begin{cor}\label{cor:CMlift}
Let $A_0/\F_q$ be an abelian variety such that the characteristic polynomial $h$ of its Frobenius endomorphism $\pi_{A_0}$ is squarefree, with CM by $L = \Q[x]/(h)$ through a $p$-adic CM-type $\Phi$.
Suppose that $(L,\Phi)$ satisfies the generalized residual reflex field condition of Definition~\ref{def:RRC} with respect to $\pi_{A_0}$ .
Then there exists an abelian variety $A'_0/\F_q$ such that $A'_0 \sim_L A_0$ and such that $\mathcal{O}_L \subseteq \mathrm{End}(A'_0)$, which has a CM-lifting $A'$ over a number field $E'$ satisfying $\mathcal{O}_L \subseteq \mathrm{End}(A')$.
In particular, after an $L$-linear isogeny the variety $A_0/\F_q$ admits a CM-lifting to a $p$-adic field, namely, to the completion of $E'$ at the induced $p$-adic place.
\end{cor}

\begin{remark}\label{rem:domain}
All of Propositions~\ref{prop:ord_lift} and \ref{prop:almord_lift} and Theorem~\ref{thm:CCOlift} provide CM-liftings of abelian varieties over finite fields to characteristic zero. In fact, all results provide liftings to $p$-adic fields: namely,  
Theorem~\ref{thm:CCOlift}
lifts to a number field and hence to its completion, which is a $p$-adic field, and Propositions~\ref{prop:ord_lift} and \ref{prop:almord_lift} lift to a finite extension of the Witt vectors $W(\F_q)$ and hence to its fraction field, which is also a $p$-adic field.
\end{remark}

\begin{remark}\label{rem:RCC}
While Theorem~\ref{thm:CCOlift} shows that the residual reflex condition is sufficient for a CM-lifting (up to isogeny) to a normal domain to exist, this condition is also necessary when the variety is simple, see \cite[2.1.4, 2.1.5]{chaiconradoort14}. 
In particular, simple ordinary abelian varieties and simple almost-ordinary abelian varieties in odd characteristic, which admit CM-liftings by Propositions~\ref{prop:ord_lift} and \ref{prop:almord_lift}, all satisfy the residual reflex condition.
\end{remark}
We conclude this section with some lifting results that will allow us, in Proposition~\ref{prop:CCO_effective}, to improve on Corollary \ref{cor:CMlift}.

Let $A_0$ be any abelian variety over a field $\F_q$.
The Dieudonn\'e module $T_p(A_0)$ splits into a direct sum of an \'etale part and a local part:
\begin{equation}\label{eq:dieudonne_dec}
    T_p(A_0) = T^{\mathrm{\acute{e}t}}_p(A_0) \oplus T^{\mathrm{loc}}_p(A_0).
\end{equation} 
An isogeny $f:A_0\to B_0$ of abelian varieties over $\F_q$ induces an injective morphism $f_p:T_p(B_0)\to T_p(A_0)$ 
which respects the decomposition \eqref{eq:dieudonne_dec}.

\begin{lemma}[{\cite[Proof of Thm 5.2]{Wat69}}]\label{lemma:Tploc_isom}
If $f:A_0\to B_0$ is a separable isogeny between abelian varieties over $\F_q$,
then the induced map $f_p^{\mathrm{loc}}:T^{\mathrm{loc}}_p(B_0)\to T^{\mathrm{loc}}_p(A_0)$ is an isomorphism.
\end{lemma}

\begin{prop}\label{prop:sepisoglifting}
Let $A_0$ and $A'_0$ be abelian varieties over $\F_q$ with CM by an algebra~$L$.
Assume that~$A_0$ admits a CM-lifting $\mathcal{A}$ to $\cR$ (cf. Definition \ref{def:CMlift}) such that $\End_L(A_0)$ also lifts.
Moreover, assume that there exists a separable isogeny $f:A_0\to A'_0$.
Then $A'_0$ also admits a CM-lifting to $\cR$ such that $\End_L(A'_0)$ also lifts. In particular, if $\End(A'_0)=\End_L(A_0')$ then such a lifting is canonical. 
\end{prop}

\begin{proof}
By the Serre-Tate deformation theorem (cf. \cite[Theorem 1.2.1]{katz}), there is a lifting~$\mathcal{A}[p^\infty]$ of the $p$-divisible group $A_0[p^\infty]$ to $\mathcal{R}$; in particular, there is a lifting of its local part~$A_0[p^\infty]^{\mathrm{loc}}$.
Further, by \cite[18.3.2]{EGAIV.4}, there is a lifting $\cA'[p^\infty]^{\mathrm{\acute{e}t}}$ of the \'etale part of $A'_0[p^\infty]$ to~$\mathcal{R}$.
Moreover, both lifting results are equivalences of categories.
Define 
\[ G := \cA'[p^\infty]^{\mathrm{\acute{e}t}} \oplus \cA[p^\infty]^{\mathrm{loc}}. \]
By Lemma \ref{lemma:Tploc_isom}, the separable isogeny $f:A_0\to A'_0$ induces an isomorphism of Dieudonn{\'e} modules $T_p^{\mathrm{loc}}(A'_0)\simeq T_p^{\mathrm{loc}}(A_0)$.
Hence, by the categorical equivalence between $p$-divisible groups and Dieudonn\'e modules, the reduction $G_0$ of $G$ to $\F_q$ satisfies
\[ G_0 = A'_0[p^\infty]^{\mathrm{\acute{e}t}} \oplus A_0[p^\infty]^{\mathrm{loc}} \simeq A'_0[p^\infty]. \]
That is, we can lift $A'_0[p^\infty]$ to $\cR$. Therefore, again by the Serre-Tate deformation theorem $A'_0$ lifts to $\cR$, as well.
By fully faithfulness of the lifting constructions, we also have that 
\[ \End_L( A'_0[p^\infty] ) \simeq \End_L(A'_0[p^\infty]^{\mathrm{\acute{e}t}}) \oplus \End_L(A_0[p^\infty]^{\mathrm{loc}})\]
lifts to $\cR$. Hence, we conclude that
\[ \End_L(A'_0) \simeq \End_L(\cA').\]
\end{proof}

\section{Algebraic descriptions of abelian varieties} \label{sec:cats}
In this section, we describe an algebraic construction and two categorial equivalences which we will use to study abelian varieties.
In Subsection~\ref{ssec:unif} we explain how to associate a fractional ideal to any CM abelian variety over a $p$-adic field. 
On the other hand, the equivalences in Subsections~\ref{ssec:CS} (from \cite{CentelegheStix15}) and \ref{ssec:Stefano} (from \cite{MarAbVar18}) together yield a functor $\mathcal{G}$ from certain abelian varieties over finite fields to certain fractional ideals. In Subsection~\ref{ssec:compred}, we discuss the compatibility of these descriptions under reductions.

\subsection{Complex uniformization and fractional ideals}\label{ssec:unif}

In this subsection, we consider abelian varieties over $p$-adic fields with complex multiplication, and provide an association with fractional ideals.\\

Let $A$ be an abelian variety defined over a $p$-adic field $K$ and with CM by an algebra~$L$ through a CM-type $\Phi$.
Fix an inclusion $i: L \hookrightarrow \mathrm{End}^0(A)$ and identifications $K \hookrightarrow \overline{K} \simeq \mathbb{C}$, and consider the complex abelian variety $A_{\mathbb{C}} := A \otimes \C$.
By complex uniformization, there exists a fractional $S$-ideal $I$ in~$L$, where $S := i^{-1}(\mathrm{End}(A))$, such that
\begin{equation}\label{eq:Cunif}
A_{\mathbb{C}}(\mathbb{C}) \simeq \mathbb{C}^g/\Phi(I).
\end{equation}
In particular, $S = (I:I) := \{ x \in L: xI \subseteq I\}$.

Let $A'$ be a second abelian variety over $K$ with the same CM-type $\Phi$ and fix an inclusion $i': L \hookrightarrow \End^0(A')$, so that $A'_{\mathbb{C}}(\mathbb{C}) \simeq \mathbb{C}^g/\Phi(I')$ for some fractional $S'$-ideal~$I'$ in $L$, with $S' := (i')^{-1}(\End(A'))$. From the description of $A_{\mathbb{C}}(\mathbb{C})$ and $A'_{\mathbb{C}}(\mathbb{C})$ 
as tori we get a bijection 
\begin{equation*}
    \Hom_L(A_{\mathbb{C}},A'_{\mathbb{C}}) \leftrightarrow (I': I):= \{ x \in L : xI \subseteq I' \}. 
\end{equation*}

We now use the language of ideal multiplications and transforms, see e.g.~\cite[Chapter 3.2, p.57]{lang}, to give another description of $\Hom_L(A,A')$ that will be used in Lemma~\ref{lem:HomAA'red} and Proposition~\ref{prop:redhomS}. 
Put $\mathfrak{a} = (I : I')$ and assume that $(S:\mathfrak{a})I=I'$.
Then there exists an $\mathfrak{a}$-multiplication $\lambda=\lambda_{\mathfrak{a}}: A_{\mathbb{C}} \to A'_{\mathbb{C}}$.
In other words, $A'_{\mathbb{C}}$ may be identified with the $\mathfrak{a}$-transform of $A_{\mathbb{C}}$.
Assume further that $\mathfrak{a} = (I : I')$ is invertible in $S$, that is, $(S:\mathfrak{a})\mathfrak{a}=S$. It then follows that 
\begin{equation}\label{eq:samemultring}
S'=(I':I')=((S:\mathfrak{a})I:(S:\mathfrak{a})I)=(I:I)=S,    
\end{equation}
see \cite[Proposition~4.1]{MarICM18}, and that $(S:\mathfrak{a})=\mathfrak{a}^{-1}=(I':I)$, see \cite[Corollary~4.5]{MarICM18}. 
We again obtain a bijection
\begin{equation}\label{eq:Cunifhom}
    \Hom_L(A_{\mathbb{C}},A'_{\mathbb{C}}) = \lambda \circ i\left(\mathfrak{a}^{-1}\right) \leftrightarrow (I':I),
\end{equation}
by arguing as in the proof of \cite[Chapter 3, Proposition 2.4, p.58]{lang}. In that proof it is assumed that $S$ is maximal, a condition that can be replaced with
$\mathfrak{a}$ being invertible in $S$. It follows from Lemmas~\ref{lem:KtoKbar} and~\ref{lem:KbartoC} that also
\begin{equation}\label{eq:Cunifhom2}
    \Hom_L(A,A') \leftrightarrow (I':I).
\end{equation}
\begin{df}\label{def:It}
\begin{enumerate}
    \item For $I\in \cI(R_{w})$ denote by $I^t$ the \emph{trace dual ideal} of $I$, defined by
\[ 
I^t :=\set{ a \in L : \Tr_{L/\Q}(aI)\subseteq \Z }. 
\]
We see that $I^t$ is also a fractional $R_{w}$-ideal.
\item 
Further denote
\[ 
\overline{I} := \{ \bar{x} : x \in I \},
\]
where $x \mapsto \overline{x}$ denotes the involution of $L$.
Clearly, $\overline{I}$ is again a fractional $R_w$-ideal.
\end{enumerate}
\end{df}

\begin{remark}
If $A' = A^{\vee}$ is the dual variety of $A$, we have that 
\begin{equation}\label{eq:Cunifvee}
A^{\vee}_{\mathbb{C}}(\mathbb{C}) \simeq \mathbb{C}^g/\Phi(\overline{I}^t)
\end{equation}
is the dual of the complex torus, by \cite[II.9, p.86]{Mum08}.
If $S = i^{-1}(\End(A))$ satisfies $S = \overline{S}$, then Equation~\eqref{eq:Cunifhom2} reads
\begin{equation*}
    \Hom_L(A,A^{\vee}) \leftrightarrow (\overline{I}^t: I).
\end{equation*}
\end{remark}

With this in mind, we define the following notation.

\begin{df}\label{def:H}
For $A$ as above, satisfying Equation~\eqref{eq:Cunif}, we write 
\begin{equation}\label{eq:H}
\mathcal{H}(A) = I.
\end{equation}
For abelian varieties $A, A'$ over $K$ with the same CM-type $(L, \Phi)$ and such that $\mathcal{H}(A) = I$ and $\mathcal{H}(A')=I'$, in view of~\eqref{eq:Cunifhom2} we define
\begin{equation}\label{eq:Hhom}
\mathcal{H}(\Hom_L(A,A')) := \Hom_L(\mathcal{H}(A), \mathcal{H}(A')) = (I':I).
\end{equation}
\end{df}

\begin{remark}\label{rem:H}
The ideal $I$ in \eqref{eq:H}, and hence the notation $\mathcal{H}$, are well-defined up to $L$-isomor-phism, since \eqref{eq:Cunif} is so. In addition, since the construction of the dual abelian variety is compatible with base change, Equation~\eqref{eq:Cunifvee} implies that $\mathcal{H}(A^{\vee}) = \overline{I}^t$. 
Thus, despite the fact that the association $A \mapsto I$ is not categorical, it will suffice for our purpose of computing polarizations.
\end{remark}
\subsection{The Centeleghe-Stix equivalence}\label{ssec:CS}

Before stating the equivalence, we collect some definitions and notation from \cite{CentelegheStix15}, which will also be used in the next subsections.

\begin{notn}\label{not2}
Let $\AV(q)$ denote the category of abelian varieties over the finite field $\F_q$ of characteristic $p$.

For $A \in \AV(q)$ of dimension~$g$, let $h = h_A$ be the characteristic polynomial of its Frobenius endomorphism. 
Then $h \in \mathbb{Z}[x]$ has degree $2g$, and the set $w = w_A$ of its complex roots consists of \emph{Weil $q$-numbers}, i.e., algebraic integers with complex absolute value $\sqrt{q}$ under any complex embedding (up to conjugacy).
By Honda-Tate theory, every such $h$ is the characteristic polynomial of Frobenius for some $A \in \AV(q)$.

Moreover, let $\AV_h(q)$ be the full subcategory of $\AV(q)$ consisting of abelian varieties $A$ whose Frobenius endomorphism $\pi_A$ has characteristic polynomial~$h$. Then the objects of $\AV_h(q)$ form exactly the isogeny class corresponding to $h$; we also call $\AV_h(q)$ \emph{the isogeny class}
corresponding to the \emph{Weil ($q$-)polynomial}~$h$. 

For a set $w$ of Weil $q$-numbers, we also define the full subcategory $\AV_w(q)$ of $\AV(q)$ consisting of abelian varieties $A$ for which $w_A \subseteq w$.

Throughout this subsection (and the next), we specialize to prime fields, i.e., we consider $\F_q = \F_p$.
Let $\AVcs{p}$ be the full subcategory of $\AV(p)$ whose objects $A$ satisfy that $w_A$ does not contain the real Weil $p$-numbers $\pm \sqrt{p}$.
For a Weil $p$-polynomial $h$ with roots $w \not\supseteq \{ \pm \sqrt{p} \}$, we have the inclusions of subcategories $\AV_h(p) \subseteq \AV_w(p) \subseteq \AVcs{p}$.

Let $w$ be any set of non-real Weil $p$-numbers, which we identify with the set of roots of some Weil $p$-polynomial $h$, and choose representatives $w = \{\pi_1, \ldots, \pi_r \}$.
Define the ring homomorphism
\[
\begin{split}
\Z[x,y]/(xy-p) & \to \prod_{i=1}^r \Q(\pi_i), \\
x & \mapsto (\pi_1, \ldots, \pi_r).
\end{split}
\]
It factors through a minimal central order denoted by $R_w$, cf.~\cite[Definition 2]{CentelegheStix15}.

For the remainder of this subsection, we only consider Weil polynomials $h$ satisfying the following condition:
\begin{equation*}
\text{\bfseries $(\dagger)$} \qquad \qquad h \text{ is a squarefree Weil ${p}$-polynomial that does not have real roots.}
\end{equation*}
Squarefree means that $h$ is a finite product of distinct irreducible monic polynomials. In this case, $\AV_h(p) =  \AV_w(p)$.

Denote by $L = \Q[x]/(h)$ the {\'e}tale algebra over $\Q$. Then equivalently, $R_w$ is the order in~$L$ generated by $\pi = x \pmod{h}$ and $p/\pi$, so we identify $\pi$ with the image of $F$ and $p/\pi$ with that of $V$. 
The order $R_w$ has finite index in $\cO_L$, the integral closure of $\Z$ in $L$. 
\end{notn}

We now briefly recall the main definitions and results of \cite{CentelegheStix15}.

\begin{df}\label{def:Mpcs}
Define the category
\[
\Mcs{p}:=\set{
\parbox[p]{30em}{$(T,F)$, where $T$ is a free finitely generated $\Z$-module and \\
$F\in \End_\Z(T)$ is such that\\
- $F\otimes \Q$ is semisimple with eigenvalues of absolute value $\sqrt{p}$\\
- the characteristic polynomial of $F \otimes \Q$ does not have real roots\\
- there exists $V\in\End_\Z(T)$ such that $FV=VF=p$
}},
\]
and for each set $w$ of non-real Weil $p$-numbers, define the subcategory $\Mcs{p}_w$ whose objects satisfy that $F \otimes \Q$ has eigenvalues in $w$.
\end{df}

\begin{thm}[{\cite[Theorem 1]{CentelegheStix15}}]\label{thm:CS}
For each set $w$ of non-real Weil $p$-numbers, there exists an antiequivalence of categories
\begin{equation}\label{eq:CSw}
\cF_w: \AV_w(p) \to \Mcs{p}_w.
\end{equation}
These induce an antiequivalence of categories
\[
\cF: \AVcs{p} \to \Mcs{p}.
\]
\end{thm}

\begin{proof}[Sketch of proof]
Let $w$ be any finite set of non-real Weil $p$-numbers.
By Honda-Tate theory, there exists an isogeny class of abelian varieties $B$ with $w_B = w$. 
In particular, there exists an abelian variety $A_w$ in this class whose endomorphism ring is minimal, i.e., with $\End_{\F_p}(A_w) = \Z[\pi_A, p/\pi_A] \simeq R_w$ \cite[Proposition 21]{CentelegheStix15}. 

On $\AV_w(p)$ the subfunctor
\begin{equation}\label{eq:cFw}
\cF_w := \Hom(- , A_w)
\end{equation}
furnishes an antiequivalence between $\AV_w(p)$ and reflexive $R_w$-modules, \cite[Theorem 25]{CentelegheStix15}. The latter category is in turn equivalent with the subcategory $\Mcs{p}_w$, \cite[Proposition 14]{CentelegheStix15}; by abuse of notation we again denote the resulting functor from $\AV_w(p)$ to $\Mcs{p}_w$ by $\cF_w$.

Moreover, the varieties $A_w$ can be chosen in such a way that they yield an ind-system, so that the direct limit of the subfunctors over subsets $w$ of the set of non-real Weil $p$-numbers~$\cW(p)$
\[
T(A) := \varinjlim_{w \subseteq \cW(p)} \Hom(A,A_w)
\]
stabilizes for any $A \in \AVcs{p}$. That is, for $w(A) \subseteq w \subseteq w'$ we have $\Hom(A,A_w) \xrightarrow{\sim} \Hom(A,A_{w'})$, cf.~\cite[Proposition 26]{CentelegheStix15}.
Finally, the linear map $F := T(\pi_A): T(A) \to T(A)$ is induced from the action of $F$ in $R_w$, as $w$ ranges over the finite subsets of $\cW(p)$. 
This concludes the construction of the ind-representable functor $\cF: A \mapsto (T(A),F)$.
\end{proof}

\begin{remark}\label{rmk:leftSrightRw}
In the Proof of Theorem \ref{thm:CS} we fix an isomorphism $i_{A_w}:R_w \xrightarrow{\sim} \End(A_w)$.
For any other abelian variety $A \in \AV_w(p)$, we may and do choose an isomorphism $i_A:S\xrightarrow{\sim} \End(A)$, where $S$ is an order in $L$ containing $R_w$, such that
\begin{equation*}
i_A^{-1}(\pi_A) = \pi = i_{A_w}^{-1}(\pi_{A_w}). 
\end{equation*}
In this way the left $S$-module structure and right $R_w$-module structure on $\cF_w(A)$ are compatible.
\end{remark}

The varieties $A_w$ with minimal endomorphism ring $R_w$ play an important role in the proof of Theorem~\ref{thm:CS}. The following fact about these varieties will be crucial later.

\begin{lemma}[{\cite[Remark 23, Theorem 37]{CentelegheStix15}}]\label{lem:Awchoice}
For any finite set $w$ of non-real Weil $p$-numbers, the set of isomorphism classes of abelian varieties $A_w$ with minimal endomorphism ring $R_w$ is equipped with a free and transitive action by $\mathrm{Pic}(R_w)$.\qed
\end{lemma}

Consider the functor $\cF_w$ in \eqref{eq:CSw}, which depends on an abelian variety $A_w$ with minimal endomorphism ring $R_w$ through \eqref{eq:cFw}. 
We may \emph{modify} the functor~$\cF_w$ by choosing instead another abelian variety $A'_w$ with endomorphism ring~$R_w$. 
Denoting the action given in Lemma~\ref{lem:Awchoice}  by a tensor product, we get that there exists a choice of $A'_w = J \otimes_{R_w} A_w$ for each ideal class $[J] \in \mathrm{Pic}(R_w)$. We prove the following lemma for the modified functor.

\begin{lemma}\label{lem:cFmod}
Let $h$ be a Weil polynomial as in $(\dagger)$ with set of roots $w$.
Let $J$ be an invertible ideal with $[J] \in \Pic(R_w)$,
and choose $A_w$ with minimal endomorphism ring $R_w$.
Modifying the functor~$\cF_w$ by replacing $A_w$ with another abelian variety $A'_w = A_w \otimes_{R_w} J$ with minimal endomorphism ring gives a new functor~$\cF'_w$, which satisfies $\cF_w'(A) \simeq J \otimes_{R_w} \cF_w(A)$ for any $A \in \AV_w$.
\end{lemma}
\begin{proof}
The result follows from Equation~\eqref{eq:cFw} and the fact that
    \[
    \Hom(A,A_w) \simeq \Hom(A,J \otimes_{R_w} A'_w) \simeq J \otimes_{R_w} \Hom(A,A'_w),
    \]
    cf.~\cite[Proposition 34]{CentelegheStix15}.
\end{proof}

\begin{remark}
There is a relation between the tensor construction and $J$-multiplication. Namely, let $A$ be a CM abelian variety over a field with $R=\End(A)$ and let $J$ be an \emph{invertible} $R$-ideal. If $A \to A'$ is a $J$-multiplication, then $A'$ is isomorphic to $J^{-1} \otimes_{R} A$. This follows from \cite[Corollary A.4]{Wat69} together with the fact that $J \otimes_R N \cong \Hom_R(J^{-1},N)$ for any $R$-module $N$.
\end{remark}

\subsection{Finite fields and fractional ideals}\label{ssec:Stefano}

Using Notation~\ref{not2} again, in particular, specializing to finite prime fields $\mathbb{F}_p$ and using condition~$(\dagger)$, we now recall the main definitions and results of \cite{MarAbVar18}.

\begin{df}\label{def:IR}
Define the category
    \[
    \cI(R_w) := \Big\{ \text{ fractional $R_w$-ideals } \Big\},
    \]
    whose morphisms are $R_w$-linear morphisms.
\end{df}

\begin{thm}[{\cite[Corollary 4.4]{MarAbVar18}}] \label{thm:Mar}
For any Weil polynomial $h$ as in $(\dagger)$, there exists an antiequivalence of categories 
\[
\cG: \AV_h(p) \to \cI(R_w).
\]
The functor $\cG$ is linear over $L = \Q[x]/(h)$.
\end{thm}

The construction of the functor consists of choosing, for a choice of $A_w$ in~\eqref{eq:cFw}, an embedding into $L$ of $\mathcal{G}(A) = \cF_w(A)$ as an $R_w$-module.
In Subsection~\ref{ssec:compred}, we will discuss this embedding in much more detail.

Recall from Lemma~\ref{lem:cFmod} and the comments preceding it that we may modify the functor $\cF_w$ by choosing a different abelian variety $A_w$ with endomorphism ring $R_w$. Such a modification also affects the functor $\cG$.

\begin{lemma}\label{lem:cGEnd}
Let $h$ be a Weil polynomial satisfying $(\dagger)$ with set of roots $w$.
Let $J$ be an invertible ideal with $[J] \in \Pic(R_w)$, let $A \in \AV_w$, and choose $A_w$ with minimal endomorphism ring~$R_w$.
Modifying the functor~$\cF_w$ by replacing $A_w$ with another abelian variety $A'_w = A_w \otimes_{R_w} J$ with minimal endomorphism ring is equivalent to replacing $\cG(A)$ with $J \cG(A)$.
\end{lemma}

\begin{remark}\label{rmk:Ghom}
Let $A, A' \in \AV_w(p)$ with CM by an algebra $L$. 
Applying the functor $\mathcal{G}$, we obtain fractional ideals $\cG(A)$ and $\cG(A')$, which are respective embeddings into $L$ of $\Hom(A,A_w)$ and $\Hom(A',A_w)$. 
Moreover, we obtain
\begin{equation}\label{eq:Ghom}
\cG(\mathrm{Hom}(A,A')) = (\cG(A):\cG(A')) 
\end{equation}
as an embedding of $\Hom(\Hom(A',A_w), \Hom(A,A_w))$ into $L$. 

Note that by Lemma~\ref{lem:cGEnd}, Equation~\eqref{eq:Ghom}, and hence the ideal $(\cG(A):\cG(A'))$, is independent of the choice of functor $\cG$; a different choice of functor $\cG'$ would yield $\cG'(A) = J \cG(A)$ and $\cG'(A') = J \cG(A')$ for some invertible ideal $J$ with $[J] \in \Pic(R_w)$, and hence 
\[
\cG'(\Hom(A,A')) = (\cG'(A) : \cG'(A')) = (J\cG(A) : J\cG(A')) = (\cG(A):\cG(A')).
\]
Since $\cG$ is an equivalence, there is a bijection $\Hom(A,A') \leftrightarrow (\cG(A):\cG(A'))$.
\end{remark}

\begin{remark}\label{rem:Endideal}
As a special case of Remark~\ref{rmk:Ghom}, we see that
\[
\cG(\mathrm{End}(A)) = (\cG(A):\cG(A')) \hspace{1cm} \text{ and } \hspace{1cm} \cG(\mathrm{Aut}(A)) = (\cG(A):\cG(A'))^{*},
\]
cf.~\cite[Corollary 4.6]{MarAbVar18}. 
\end{remark}

The following proposition describes the isomorphism class of the image of the dual variety under~$\cG$. 
Recall from Definition~\ref{def:It} that for a fractional ideal $I$ we denote its trace dual by $I^t$ and its image under the CM-involution of $L$ by $\bar{I}$.

\begin{prop}\label{prop:CSdual}
There exists an invertible ideal $H$ with $[H] \in \mathrm{Pic}(R_w)$ such that 
\begin{equation}\label{eq:globaliso}
\mathcal{G}(A^{\vee}) \simeq H \ \overline{\mathcal{G}(A)}^t  
\end{equation}
for any $A \in \AV_w(p)$.
\end{prop}

\begin{proof}
Fix an abelian variety $A_w$ with minimal endomorphism ring $R_w$. We will repeatedly use that $R_w = \overline{R_w}$ and $R_w \simeq R_w^t$, cf.~\cite[Proof of Proposition 9.4]{Howe95}. 

By \cite[Proposition 21]{CentelegheStix15}, the $\ell$-adic Tate module of $A_w$ satisfies
\begin{equation}\label{eq:TlAwI}
T_{\ell}(A_w) \simeq R_w \otimes \mathbb{Z}_{\ell}
\end{equation}
for all $\ell \neq p$, while at $\ell = p$ its Dieudonn{\'e} module satisfies
\begin{equation}\label{eq:TpAwI}
T_{p}(A_w) \simeq R_w^t \otimes \mathbb{Z}_{p} \simeq R_w \otimes \mathbb{Z}_p.
\end{equation}
For the remainder of the proof, we fix the isomorphisms of~\eqref{eq:TlAwI} and~\eqref{eq:TpAwI}.

First consider $\ell \neq p$. Let $M := \cF_w(A)$.
Combining \eqref{eq:TlAwI} with \cite[Proposition 27]{CentelegheStix15}, for any $A \in \AV_w(p)$ we have an identification
\begin{equation}\label{eq:TlA}
T_{\ell}(A) \simeq \Hom_{R_w \otimes \mathbb{Z}_{\ell}}(M \otimes \mathbb{Z}_{\ell}, T_{\ell}(A_w)).
\end{equation}
Furthermore, for any prime $\ell \neq p$ we have
\begin{equation}\label{eq:LHS}
\begin{split}
\cF_w(A^{\vee}) \otimes \Z_{\ell} & = \Hom_{\F_p}(A^{\vee},A_w) \otimes \Z_{\ell} \simeq \Hom_{\Z_{\ell}[\mathrm{Gal}_{\F_p}]}(T_{\ell}(A^{\vee}), T_{\ell}(A_w) ) \\
& \simeq \Hom_{R_w \otimes \Z_{\ell}}( T_{\ell}(A^{\vee}), T_{\ell}(A_w) ) \\
& \simeq \Hom_{R_w \otimes \Z_{\ell}}( \overline{\Hom_{\Z_{\ell}}(T_{\ell}(A), \Z_{\ell})}, T_{\ell}(A_w) ) \\
& \simeq \Hom_{R_w \otimes \Z_{\ell}}( \overline{\Hom_{\Z_{\ell}}(\Hom_{R_w \otimes \mathbb{Z}_{\ell}}(M \otimes \mathbb{Z}_{\ell}, T_{\ell}(A_w)), \Z_{\ell})}, T_{\ell}(A_w) ) \\
& \simeq \Hom_{R_w \otimes \mathbb{Z}_{\ell}}( \overline{\Hom_{\Z_{\ell}}(\Hom_{R_w \otimes \mathbb{Z}_{\ell}}(M \otimes \mathbb{Z}_{\ell}, R_w \otimes \mathbb{Z}_{\ell}), \Z_{\ell})}, R_w \otimes \mathbb{Z}_{\ell} ) \\
& \simeq \Hom_{R_w}( \overline{\Hom_{\Z}(\Hom_{R_w}(M, R_w), \Z)}, R_w) \otimes \mathbb{Z}_{\ell};
\end{split}
\end{equation}
the first line is Tate's theorem \cite{Tate66}, the second is \cite[Remark 20]{CentelegheStix15}, the third is an identification $T_{\ell}(A^{\vee}) \simeq T_{\ell}(A)^t(1) \simeq T_{\ell}(A)^t$ of $R_w$-modules induced from a non-canonical isomorphism $\mathbb{Z}_{\ell}(1) \simeq \mathbb{Z}_{\ell}$ of group 
schemes (which does not affect the module structure), the fourth is \eqref{eq:TlA}, and the fifth is \eqref{eq:TlAwI}.

We remark that for any ideals $I$ and $J$ in $L$, with $J^t$ invertible, we have $(J:I)=(I^t:J^t)=(J^t)^{-1}I^t$.
Moreover, applying the CM-involution of $L$ commutes with taking inverses and trace duals.
Applying these rules, we obtain
\begin{equation}\label{eq:LHSglobal}
\begin{split}
 (J^t)^{-1} \overline{(J^t)^{-1}} \overline{I}^t & =
 (\overline{(J^t)^{-1}} \overline{I}^t : J^t) =
 ((\overline{I}^t : \overline{J}^t) : J^t) = \\
 & = ((\overline{J} : \overline{I}) : J^t) = 
 (J : \overline{(J:I)}^t) =
 \\
& = \Hom_{R_w}( \overline{\Hom_{\Z}(\Hom_{R_w}(I, J), \Z)}, J ).
\end{split}
\end{equation}
Recall that $R_w$ is a Gorenstein order. Hence $R_w^t$ is an invertible fractional $R_w$-ideal, i.e., $R_w^t$ is locally principal.
By choosing embeddings into $L$ such that \eqref{eq:LHS} becomes an equality, 
we obtain from \eqref{eq:LHSglobal} (with $I = \mathcal{G}(A)$ and $J = R_w$) that \begin{equation}\label{eq:LHSlocal}
\mathcal{G}(A^{\vee}) \otimes \mathbb{Z}_{\ell} \simeq \overline{\mathcal{G}(A)}^t \otimes \mathbb{Z}_{\ell}
\end{equation}
for any $\ell \neq p$. 
We stress that the local isomorphism in \eqref{eq:LHSlocal} is 
multiplication by $(x_0\bar{x_0})^{-1}$, where $x_0$ is the local generator at $\ell$ of $R_w^t$; in particular, it does not depend on the choice of $A \in \AV_w(p)$.

Now we consider the prime $p$. Combining \eqref{eq:TpAwI} with \cite[Proposition 28]{CentelegheStix15} yields an identification
\begin{equation*}
T_p(A) \simeq (\cF_w(A) \otimes \mathbb{Z}_p) \otimes_{(R_w \otimes \mathbb{Z}_p)} T_p(A_w) \simeq (M \otimes_{R_w} R_w^t) \otimes \Z_p.
\end{equation*}
Instead of \eqref{eq:LHS} we find by analogous arguments (and by contravariance of the Dieudonn{\'e} functor) that
\begin{equation}\label{eq:LHSp}
\begin{split}
\cF_w(A^{\vee}) \otimes \Z_{p} & = \Hom_{\F_p}(A^{\vee},A_w) \otimes \Z_{p} \simeq \Hom_{\mathcal{D}_{\F_p}}(T_{p}(A_w), T_{p}(A^{\vee}) ) \\
& \simeq \Hom_{R_w \otimes \Z_{p}}( T_{p}(A_w), T_{p}(A^{\vee}) ) \\
& \simeq \Hom_{R_w \otimes \Z_{p}}( T_p(A_w), \overline{\Hom_{\Z_p}(T_p(A), \Z_{p})} ) \\
& \simeq \Hom_{R_w \otimes \Z_{p}}(T_p(A_w), \overline{\Hom_{\Z_{p}}((M \otimes_{R_w} R_w^t) \otimes \Z_p, \Z_{p})}) \\
& \simeq \Hom_{R_w \otimes \Z_{p}}(R_w^t \otimes \Z_p, \overline{\Hom_{\Z_{p}}((M \otimes_{R_w} R_w) \otimes \Z_p, \Z_{p})}) \\
& \simeq \Hom_{R_w}(R_w^t, \overline{\Hom_{\Z}((M \otimes_{R_w} R_w^t), \Z)}) \otimes \Z_p;
\end{split}
\end{equation}
the first line is due to Tate, cf.~\cite[Section 1.2]{Wat69}, the second is \cite[Remark 20]{CentelegheStix15} again, the third is an identification $T_{p}(A^{\vee}) \simeq T_{p}(A)^t(1) \simeq T_{p}(A)^t$ induced from an isomorphism $\mathbb{Z}_p(1) \simeq \mathbb{Z}_p$, the fourth follows from \cite[Proposition 28]{CentelegheStix15}, and the fifth is \eqref{eq:TpAwI}.

We also remark that for any ideals $I$ and $J$ in $L$, with $J^t$ invertible, we have that
$(J^t I)^t = (J:I) = (I^t:J^t) = (J^t)^{-1} I^t$. 
Using these rules, again together with the fact that the CM-involution commutes with inverses and trace duals, we obtain, cf.~\eqref{eq:LHSglobal},
\begin{equation}\label{eq:LHSglobalp}
\begin{split}
(J^t)^{-1} \overline{(J^t)^{-1}} \overline{I}^t & = 
(J^t)^{-1}( \overline{I}^t : \overline{J}^t)=
(( \overline{I}^t : \overline{J}^t) : J^t) = \\
& = ( (\overline{I J^t})^t : J^t) = \Hom_{R_w}(J^t, \overline{ \Hom_{\Z}( I \otimes_{R_w} J^t, \Z)} ).
\end{split}
\end{equation}
So embedding $\cF_w(A)=M$ into $L$ as above, we obtain from \eqref{eq:LHSp} and \eqref{eq:LHSglobalp} (with $I = \mathcal{G}(A)$ and $J = R_w$) that
\begin{equation}\label{eq:LHSlocalp}
\mathcal{G}(A^{\vee}) \otimes \mathbb{Z}_{p} \simeq \overline{\mathcal{G}(A)}^t \otimes \mathbb{Z}_{p}.
\end{equation}
As in Equation~\eqref{eq:LHSlocal}, the local isomorphism in ~\eqref{eq:LHSlocalp} is the composition of the isomorphisms from \eqref{eq:LHSp} with multiplication by $(x_0\bar{x_0})^{-1}$, where $x_0$ is the local generator at $p$ of the invertible $R_w$-ideal $R_w^t$. Thus, again it does not depend on the choice of $A \in \AV_w(p)$.

Equations~\eqref{eq:LHSlocal} and~\eqref{eq:LHSlocalp} are equivalent to the statement that $\mathcal{G}(A^{\vee})$ and $\overline{\mathcal{G}(A)}^t$ 
are locally isomorphic at every prime ideal $\p$ of $\End(A^\vee)$, see \cite[Section 5]{LevyWiegand85}.
This is in turn equivalent to the existence of an ideal $H'$ which is invertible in $\End(A^{\vee}) = (\overline{\mathcal{G}(A)}:\overline{\mathcal{G}(A)})$ such that 
\[
\mathcal{G}(A^{\vee}) = H'\ \overline{\mathcal{G}(A)}^t,
\]
see \cite{dadetz62} for the case when $\End(A^\vee)$ is an integral domain and \cite[Proposition 4.1]{MarICM18} for a general order in an \'etale algebra.
Replacing $H'$ with an isomorphic ideal $H''$ which is coprime to $(R_w: \mathcal{O}_L)$ if necessary, cf.~\cite[Corollary 1.2.11]{cohenadv00}, we set $H := H'' \cap R_w$; then $H$ is invertible in $R_w$ by construction  and satisfies \eqref{eq:globaliso}, as required. As we observed throughout the proof, the isomorphism class of~$H$ does not depend on the variety $A \in \AV_w(p)$.
\end{proof}

\subsection{Compatibility under reduction}\label{ssec:compred}

The functor $\cG$ identifies abelian varieties over finite fields with fractional ideals, while $\mathcal{H}$ makes such an identification for abelian varieties over $p$-adic fields. The reduction map provides a relation between abelian varieties (and homomorphisms between them) in characteristic zero and those in characteristic~$p$.  
In this subsection, we prove in Propositions~\ref{prop:redhomS} under some conditions how the corresponding fractional ideals are related. 

\begin{notn}\label{not3}
Let $K$ be a $p$-adic field with residue field $k$. For an abelian variety $A$ over $K$ with good reduction, we denote its reduction by $A_k = A \otimes k$. 
\end{notn}

If $A,A'$ are abelian varieties over $K$ with good reduction and with the same endomorphism ring, whose respective reductions $A_k, A'_k$ are elements of $\AV_h(p)$ for a polynomial~$h$ satisfying $(\dagger)$ (cf.~Notation~\ref{not2}), then we recall from the previous subsections (cf.~\eqref{eq:Cunifhom2} and~\eqref{eq:Hhom}, and Remark~\ref{rmk:Ghom}) that there are bijections 
\[
\Hom_L(A,A') \leftrightarrow \mathcal{H}(\Hom_L(A,A')) \quad \text{ and } \quad \Hom_L(A_k, A'_k) \leftrightarrow \cG(\Hom_L(A_k,A_k')).
\]
A further bijection is given by the following lemma.

\begin{lemma}\label{lem:HomAA'red}
Let $A, A'$ be abelian varieties over $K$ with good reduction, and CM by an algebra~$L$ through a CM-type $\Phi$.
Denote their respective reductions by $A_k$ and $A'_k$.
Suppose that $\End(A) \simeq \End(A_k)$ and that $\End(A') \simeq \End(A'_k)$.
Define $I = \mathcal{H}(A)$ and $I' = \mathcal{H}(A')$ and suppose that $(I:I')$ is invertible and that $(I':I)I=I'$. 
Then $\End(A) \simeq \End(A')$ and the $L$-linear reduction map 
\begin{equation*}
\Hom_L(A,A')  \hookrightarrow \Hom_L(A_k,A'_k)
\end{equation*}
is a bijection. Furthermore, both sets are in bijection with $(I':I)$.
\end{lemma}

\begin{proof}
Fix an embedding $i: L \hookrightarrow \End^0(A)$. Using ideal multiplications and transforms, in Subsection~\ref{ssec:unif} we 
constructed an isomorphism $\End(A) \simeq \End(A')$ (cf.~Equation \eqref{eq:samemultring})
and obtained
a bijection
\begin{equation}\label{eq:HomAA'red1}
\Hom_L(A,A') = \lambda \circ i\left((I': I)\right) \leftrightarrow (I': I),
\end{equation}
where $\lambda$ is an $\mathfrak{a}$-multiplication for $\mathfrak{a}=(I: I')$.
It follows from e.g.~\cite[Chapter 3, Proof of Lemma 3.1, p.61]{lang}, which immediately generalizes from the case of $S$ being maximal, that the reduction of an $\mathfrak{a}$-multiplication is again an $\mathfrak{a}$-multiplication. Denote the reduction of~$\lambda$ by $\lambda_{k}$ and fix an embedding $i_k: L \hookrightarrow \End^0(A_k)$. Since $\End(A) \simeq \End(A_k)$, we obtain a similar bijection
\begin{equation}\label{eq:HomAA'red2}
\Hom_L(A_k,A_k') = \lambda_{k} \circ i_k\left((I': I)\right) \leftrightarrow (I':I).
\end{equation}
Equations~\eqref{eq:HomAA'red1} and~\eqref{eq:HomAA'red2} show that the injective reduction map 
$\Hom_L(A,A') \hookrightarrow \Hom_L(A_k,A'_k)$
is a bijection as well, as required.
\end{proof}

Since our goal is to compute polarizations, from now on we consider the special case where $A' = A^{\vee}$. Then if $A$ reduces to $A_k$ over $\F_p$, the previous subsections yield bijections 
\begin{equation*}
\Hom_L(A,A^{\vee}) \leftrightarrow (\mathcal{H}(A^{\vee}): \mathcal{H}(A)) \quad \text{ and } \quad \Hom_L(A_k, A_k^{\vee}) \leftrightarrow (\cG(A_k) : \cG(A_k^{\vee})),
\end{equation*}
and if the hypotheses of Lemma~\ref{lem:HomAA'red} are satisfied, there is a bijection between these two colon ideals.

To characterize polarizations we do not just need bijections but $L$-linear isomorphisms. To obtain these, we determine the $L$-linear structures on the objects we are considering.

\begin{remark}\label{rem:Llinearity} \
\begin{enumerate}
    \item 
An inclusion $i:L\into \End^0(A)$ induces an inclusion $i_k:L\into \End^0(A_k)$ such that the following diagram commutes:
\begin{equation}\label{eq:LEnd0}
\begin{tikzcd}[swap,bend angle=45]
L \rar{i} \dar{=} & \End^0(A) \dar{} \\ 
L \rar{i_k} & \End^0(A_k) \\ 
\end{tikzcd}
\end{equation}
Now let $A/K$ be an abelian variety with complex multiplication by a CM-algebra~$L$.
Suppose that we $\End^0(A_k)$ is a free module of rank one over $L$, which happens in particular when $A_k$ is squarefree. Then we can identify $L = \End^0(A_k)$, and it follows from diagram~\eqref{eq:LEnd0} that we can also identify
\[
\End^0(A) = \End^0(A_k).
\]

\item Similarly, an inclusion $i:L\into \End^0(A)$ induces an $L$-linear structure $i^\vee:L\into \End^0(A^\vee)$ on the dual variety satisfying
\[
i^{-1}(\End(A)) = \overline{(i^\vee)^{-1}(\End(A^\vee))}.
\]
\item It follows from the previous two remarks that if $A$ is squarefree, then there are compatible identifications $i: L \simeq \End^0(A)$, $i_k: L \simeq \End^0(A_k), i^{\vee}: L \simeq \End^0(A^{\vee})$, and $i^{\vee}_k: L \simeq \End^0(A^{\vee}_k)$. Compatibility may be expressed by identifying the respective Frobenius endomorphisms. Through these identifications, we may also identify the respective endomorphisms rings with orders in $L$. 
\item Now consider $\Hom^0(A,A^{\vee})$; it carries both a left $L$-module structure (induced from the isomorphism $i$) and a right $L$-module structure (induced from $i^{\vee}$), which are compatible by our assumptions. The analogous statements for $\Hom^0(A_k, A^{\vee}_k)$ also hold.
For any isogenies $f:A\to A^\vee$ and $g:A_k\to A^\vee_k$, we obtain $L$-linear isomorphisms
\begin{align*}
    \Psi_f:L & \xrightarrow{\sim} \Hom_L^0(A,A^\vee), & \text{ and }  && \Psi_{g}:L & \xrightarrow{\sim} \Hom_L^0(A_k,A_k^\vee), \\
    l & \mapsto f \circ i(l) = i^\vee(l) \circ f & && l & \mapsto g \circ i_k(l) = i_k^\vee(l) \circ g.
\end{align*}
Through these isomorphisms, we may identify both 
\begin{equation*}
   \Psi_f^{-1}(\Hom_L(A,A^\vee)) \quad \text{ and } \quad \Psi_{g}^{-1}(\Hom_L(A_k,A_k^\vee)) 
\end{equation*}
with fractional ideals in $L$. Equations~\eqref{eq:HomAA'red1} and~\eqref{eq:HomAA'red2} suggest that $f = \lambda$ and $g = \lambda_{k}$ are natural choices.
\item 
Let $A_0$ and $B_0$ be squarefree abelian varieties over a finite field $k$ with CM by an algebra~$L$ so that $\End^0(A_0) = L$ and $\End^0(B_0) = L$. Suppose that there exists an isogeny $f: A_0 \to B_0$. 
Since $f \circ \pi_{A_0} \circ f^{-1} = \pi_{B_0}$, 
the $L$-linear structures of $\Hom^0(A_0, A_0^{\vee}) \simeq L$ and $\Hom^0(B_0, B_0^{\vee}) \simeq L$ may be identified as well, via the map
\begin{align*}
    \Hom^0(B_0, B_0^{\vee}) & \to \Hom^0(A_0, A_0^{\vee}) \\
\phi & \mapsto f^{\vee} \circ \phi \circ f. \nonumber
\end{align*}
\item Consider the notation $\mathcal{H}$ defined in Subsection~\ref{ssec:unif}, with $\mathcal{H}(A)=I$ and $\mathcal{H}(A^{\vee}) = \overline{I}^t$. 
Equation~\eqref{eq:Cunifhom} implies that if we take the basis element $f$ in~4. to be the $\mathfrak{a}$-multiplication $\lambda: A \to A^{\vee}$, where $\mathfrak{a}=(I:\bar{I}^t)$, and where we assume that $\mathfrak{a}$ is invertible and that $\mathfrak{a}^{-1}I=\overline{I}^t$, then $\Psi_{\lambda}$ and $\mathcal{H}$ are inverse to each other.
In particular
\[
\Psi_{\lambda}^{-1}(\Hom_L(A,A^{\vee})) = (\overline{I}^t:I) 
\]
is an equality of fractional ideals in $L$.
\end{enumerate}
\end{remark}

Recall the definition of $R_w$ from Notation~\ref{not2}.
The following proposition is a strengthening of Lemma~\ref{lem:HomAA'red} (in the case that $A' = A^{\vee}$) and Remark~\ref{rem:Llinearity}.4, using the notation of Remark~\ref{rem:Llinearity}.

\begin{prop}\label{prop:redhomS}
Let $A/K$ be an abelian variety with good reduction and with CM by an algebra $L$, and let $A_k$ denote its reduction.
Then $i^{-1}(\End(A)) = S$ is an order in~$L$; suppose that~$S$ is Gorenstein and satisfies $S = \overline{S}$.
Suppose that $i_k^{-1}(\mathrm{End}(A_k))~=~S$, as well.
Then there exists an $\alpha \in S^*$ such that the reduction map $\Hom_L(A,A^{\vee}) \to \Hom_L(A_k,A_k^{\vee})$ is multiplication by $\alpha \in S^*$.
\end{prop}

\begin{proof}
Say that $\mathcal{H}(A)=I$. Then $( \overline{I}^t: \overline{I}^t ) = \overline{S}$, which by assumption is equal to $S$. Furthermore, since $S = \overline{S}$ is assumed to be Gorenstein, it follows that $(I:\overline{I}^t)$ is invertible and that $(\overline{I}^t:I)I=\overline{I}^t$, see \cite[Proposition 2.11, Proposition 4.1, and Corollary 4.5]{MarICM18}.
Remark~\ref{rem:Llinearity}.6 then implies that 
\begin{equation}\label{eq:HPsi}
(\overline{I}^t:I) = \mathcal{H}(\Hom_L(A,A^{\vee})) = \Psi_{\lambda}^{-1}(\Hom_L(A,A^{\vee})).
\end{equation}
In addition, it follows from~\eqref{eq:HomAA'red1} and~\eqref{eq:HomAA'red2} in Lemma~\ref{lem:HomAA'red} that the reduction map
\begin{equation*}
\Hom_L(A,A^{\vee}) = \lambda \circ i\left((\overline{I}^t:I)\right) \hookrightarrow \Hom_L(A_k, A_k^{\vee}) = \lambda_{k} \circ i_k\left((\overline{I}^t:I)\right)
\end{equation*}
is an $L$-linear isomorphism, where $\lambda_{k}$ denotes the reduction of $\lambda$. 
That is,
\begin{equation}\label{eq:psipsi}
\Psi_{\lambda}^{-1}(\Hom_L(A,A^{\vee})) = \Psi_{\lambda_{k}}^{-1}(\Hom_L(A_k,A_k^{\vee})) = (\overline{I}^t:I).
\end{equation}
More precisely, we obtain a commutative diagram
\[
\xymatrix{
\Hom_L(A,A^\vee) \ar[d]^{\Psi_{\lambda}^{-1}} \ar[r]^{\sim} & \Hom_L(A_k,A^\vee_k)\ar[d]^{\Psi_{\lambda_{k}}^{-1}}\\
(\overline{I}^t:I) \ar@{.>}[r]^{\sim} & (\overline{I}^t:I)
}
\]
where the top horizontal map is the reduction map and the induced dotted arrow is an isomorphism of $(\overline{I}^t:I)$ onto itself.
Since
\[
\left((\overline{I}^t : I) : (\overline{I}^t : I) \right) = 
          \left( \overline{I}^t: I(\overline{I}^t : I)\right) = ( \overline{I}^t: \overline{I}^t ) = \overline{S} = S,
\]
this isomorphism is the multiplication-by-$\alpha$ map for some $\alpha \in S^*$. 
\end{proof}

\begin{remark} \label{rem:Gorenstein} Note that in Proposition~\ref{prop:redhomS}, the assumptions that $S$ is Gorenstein and satisfies $S = \overline{S}$ are used to show that $(I:\overline{I}^t)$ is invertible and $(\overline{I}^t:I)I=\overline{I}^t$, for any $S$-ideal $I$ such that $(\overline{I}^t: \overline{I}^t ) = S $. Hence, instead of requiring $S = \overline{S}$ to be Gorenstein, we could impose the slightly weaker conditions that $(I:\overline{I}^t)$ is invertible and $(\overline{I}^t:I)I=\overline{I}^t$. 
\end{remark}

\begin{remark}
Let $A/K$ be an abelian variety with good reduction and with CM by an algebra~$L$, and let~$A_k$ denote its reduction. Suppose that $S = i^{-1}(\End(A))$ is a Gorenstein order in $L$ which satisfies $S = \overline{S}$. By \cite[Corollary 6.2]{GorenLauter}, the condition $i^{-1}(\End(A)) = i_k^{-1}(\End(A_k))$ of Proposition~\ref{prop:redhomS} holds if the cardinality $\vert \mathcal{O}_L / S \vert$ is coprime to $p$. That is, in this case we have $i_k^{-1}(\End(A_k))~=~S$, as well.
\end{remark}

Throughout the remainder of this subsection, we specialize to prime fields $\mathbb{F}_p$ and let $h$ be a squarefree polynomial satisfying $(\dagger)$ (cf.~Notation~\ref{not2}). We now study the $L$-linear structures associated with the $L$-linear functor $\cG$, given in Theorem~\ref{thm:Mar} as an embedding into $L=\Q[x]/(h)=\Q[\pi]$ obtained from the functor $\cF_w$ of Theorem~\ref{thm:CS}.

Namely, for every abelian variety $A_0$ in the isogeny class $\AV_h(p)$ there exist $\lambda_{A_0} \in \cF_w(A_0)=\Hom(A_0,A_w)$ and an embedding $i_{A_0}:L\into \End^0(A_0)$ such that
\begin{equation}\label{eq:embFw}
    \cF_w(A_0)=\Hom(A_0,A_w) = \lambda_{A_0}\circ i_{A_0}(\cG(A_0))
\end{equation}
and
\begin{equation}\label{eq:pi_comp}
    i_{A_0}(\pi)=\pi_{A_0};
\end{equation}
cf.~Remark~\ref{rmk:leftSrightRw}.
It follows from the definition of $\cF_w$ in~\eqref{eq:cFw} that for every pair of abelian varieties $A_0$ and $A'_0$ in $\AV_h(p)$ we have
\begin{equation}\label{eq:FwHom}
\cF_w(\Hom(A_0,A'_0)) = \Hom_{R_w}(\Hom(A'_0,A_w), \Hom(A_0,A_w)).
\end{equation}
Identifying these with homomorphisms $A_0 \to A'_0$ by precomposition, we may view
\[
\cF_w(\Hom(A_0,A'_0)) \hookrightarrow \Hom(A_0,A'_0).
\]
By fully faithfulness of the equivalence $\cF_w$ this inclusion is a canonical bijection of sets.
In particular, it follows from~\eqref{eq:embFw}, \eqref{eq:pi_comp}, and \eqref{eq:FwHom} that
\[
\cG(\End(A_0)) = i_{A_0}^{-1}(\End(A_0)).
\]
Now suppose that $A_0 = A_k$ is the reduction of some $A/K$ satisfying the hypotheses of Proposition \ref{prop:redhomS}, and choose $A'_0 = A^{\vee}_k$. We prove the following lemma.

\begin{lemma}\label{lem:GhomPsi}
Let $A_k \in \AV_h(p)$ be the reduction of an abelian variety $A/K$ which satisfies the hypotheses of Proposition~\ref{prop:redhomS}.
For an appropriate choice of basis elements, the natural map
\begin{equation}\label{eq:nat_map}
    \begin{split}
        \Psi_{\lambda_{k}}^{-1}\left(\Hom_L(A_k,A_k^{\vee})\right) & \longrightarrow \cG\left(\Hom_L(A_k,A_k^\vee)\right)\\
        \Psi_{\lambda_{k}}^{-1}(f) & \longmapsto \cG(f).
    \end{split}
\end{equation}
is the identity map. 
That is, $\cG(f) = \Psi_{\lambda_{k}}^{-1}(f)$ for every $f:A_k\to A_k^\vee$.
\end{lemma}

\begin{proof}
Choosing basis elements $\lambda_{A_k} \in \Hom(A_k,A_w)$ and $\lambda_{A_k^{\vee}} \in \Hom(A_k^{\vee},A_w)$, Equation~\eqref{eq:embFw} reads
\begin{equation}\label{eq:GPsi}
    \begin{split}
        \Hom(A_k,A_w) &= \lambda_{A_k} \circ i_k(\mathcal{G}(A_k)),\\
        \Hom(A_k^{\vee},A_w) &= \lambda_{A_k^{\vee}} \circ i_k^{\vee}(\mathcal{G}(A_k^{\vee})).
    \end{split}
\end{equation}
Now we have two descriptions of $\Hom_L(A_k,A_k^\vee)$ in terms of fractional ideals in~$L$, namely, $\cG(\Hom_L(A_k,A_k^\vee)) = (\cG(A_k) : \cG(A_k^{\vee}))$ and $\Psi_{\lambda_{k}}^{-1}\left(\Hom_L(A_k,A_k^{\vee})\right) = (\overline{I}^t:I)$; the latter equality is Equation~\eqref{eq:psipsi}.
(Recall from~Remark~\ref{rmk:Ghom} that the $L$-isomorphism class of $\cG(\Hom_L(A_k, A_k^{\vee}))$ does not depend on the choice of functor $\cG$.) 
If we choose
\begin{equation}\label{eq:lambdachoice}
 \lambda_{A_k} = \lambda_{A_k^{\vee}} \circ \lambda_k
\end{equation} then it follows from~\eqref{eq:GPsi} that the natural map \eqref{eq:nat_map} is actually the identity.
Indeed, the map factors through
\begin{alignat*}{5}
    (\bar{I}^t : I) & \xrightarrow{\Psi_{\lambda_{k}}} \Hom(A_k,A_k^\vee) & \longrightarrow \cF_w\left(\Hom(A_k,A^\vee_k)\right) & \longrightarrow \left(\cG(A_k):\cG(A_k^\vee)\right)
    \\ z & \longmapsto \ \lambda_k\circ i_k(z) & y & \longmapsto  \left(g \mapsto i_k^{-1}\circ\lambda_{A_k}^{-1}\circ y \circ  \lambda_{A_k^\vee} \circ i^{\vee}_k(g)\right)
    \\ & \hspace{2.68cm}x & \longmapsto \left(\cF_w(x): h\mapsto h\circ x\right) &
\end{alignat*}
Hence, the image of $z\in (\bar{I}^t : I)$ is
\begin{align*}
     g \mapsto 
     & i_k^{-1}\circ\lambda_{A_k}^{-1}\circ \cF_w(\lambda_k\circ i_k(z)) \circ  \lambda_{A_k^\vee} \circ i^{\vee}_k(g) 
     = i_k^{-1}\circ\lambda_{A_k}^{-1}\circ  \lambda_{A_k^\vee} \circ i^{\vee}_k(g) \circ \lambda_k\circ i_k(z) \\
     & =  i_k^{-1}\circ\lambda_{A_k}^{-1}\circ  \lambda_{A_k} \circ \lambda_k^{-1} \circ i^{\vee}_k(g) \circ \lambda_k\circ i_k(z)
     = i_k^{-1}\circ \lambda_k^{-1} \circ  \lambda_k \circ i_k(g) \circ i_k(z) \\
     & =  i_k^{-1}\circ i_k(g) \circ i_k(z) = gz,
\end{align*}
that is, multiplication by $z$, as required.
\end{proof}

The following corollary of Lemma~\ref{lem:GhomPsi} is a strengthening of Proposition~\ref{prop:CSdual}.

\begin{cor}\label{cor:JJbar}
Assume that the hypotheses of Lemma~\ref{lem:GhomPsi} hold.
Then there exists an invertible ideal $J$ with $[J] \in \mathrm{Pic}(R_w)$ such that 
\begin{equation*}
\mathcal{G}(B^{\vee}) \simeq \left( J \overline{J} \right) \overline{\mathcal{G}(B)}^t  
\end{equation*}
for any $B \in \AV_w(p)$.
\end{cor}

\begin{proof}
By Proposition~\ref{prop:CSdual}, it suffices to show that the invertible ideal $H$ appearing in that proposition is of the form $H \simeq J \overline{J}$. Since $H$ is fixed for all varieties in $\AV_w(p)$, we determine it by considering the variety $A_k \in \AV_w(p)$ which is the reduction of $A/K$.
It follows from Equation~\eqref{eq:HPsi} and Lemma~\ref{lem:GhomPsi}, if we choose basis elements as in~\eqref{eq:lambdachoice}, that
\[
(\mathcal{G}(A_k): H\ \overline{\mathcal{G}(A_k)}^t) \simeq \mathcal{G}(\Hom(A_k,A_k^{\vee})) = \Psi^{-1}(\Hom(A_k,A_k^{\vee})) = (\overline{I}^t:I).
\]
Hence,
\[
 H\cdot \overline{\mathcal{G}(A_k)}^t\cdot \mathcal{G}(A_k)^t  \simeq \overline{I} I 
\]
and therefore $H \simeq J \overline{J}$ for $J= I\cdot  \left(\mathcal{G}(A_k)^t\right)^{-1}$.
\end{proof}

\begin{cor}\label{cor:Fwmod}
Assume that the hypotheses of Lemma~\ref{lem:GhomPsi} hold. Let $h$ be a Weil polynomial satisfying $(\dagger)$ (cf.~Notation~\ref{not2}), and let $J_1$ be any invertible $R_w$-ideal. Choose an abelian variety $A_w$ with minimal endomorphism ring $R_w$ which determines the functor $\cF_w$ as in \eqref{eq:cFw} and hence determines $\cG$.
Then there exists an invertible $R_w$-ideal $J_2$ such that the modifcation $\cG'$ of $\cG$, obtained by replacing $A_w$ with $A_w' = A_w \otimes_{R_w} J_2$ as in Lemma~\ref{lem:cFmod}, satisfies
\[
\cG'(A^{\vee}) \simeq \left(J_1 \overline{J_1}\right) \overline{\cG'(A)}^t
\]
for any $A \in \AV_h(p)$.
\end{cor}

\begin{proof}
By Corollary~\ref{cor:JJbar}, we have
\[
\cG(A^{\vee}) \simeq \left(J_3 \overline{J_3}\right) \overline{\cG(A)}^t,
\]
for some invertible $R_w$-ideal $J_3$. Let $J_2 = J_1 J_3^{-1}$ and let $\cG'$ be the corresponding modified functor. Then by Lemma~\ref{lem:cGEnd} we have
\[
\cG'(A^{\vee}) \simeq J_2 \cG(A^{\vee}) \simeq J_2 \left(J_3 \overline{J_3}\right) \overline{\cG(A)}^t \simeq \left( J_2 \overline{J_2}\right) \left(J_3 \overline{J_3}\right) \overline{\cG'(A)}^t = \left(J_1 \overline{J_1}\right) \overline{\cG'(A)}^t,
\]
as required.
\end{proof}

The following proposition shows how the $L$-linear structures for $\cG$ behave with respect to duals of homomorphisms.

\begin{prop}\label{prop:dual}
Fix a map $\eta \in \Hom(A_w^\vee,A_w)$.
For every $A \in \AV_h(p)$, we choose embeddings $i_A, i_{A^{\vee}}$ and basis elements $\lambda_{A}, \lambda_{A^{\vee}}$ such that 
\[
\begin{split}
    & \cF_w(A) = \Hom(A,A_w) = \lambda_A\circ i_A(\cG(A)), \\
    & \cF_w(A^\vee) = \Hom(A^\vee,A_w) = \lambda_{A^\vee}\circ i_{A^\vee}(\cG(A^\vee)),\\
    & \lambda_{A^\vee} \circ \lambda_{A}^\vee = \eta.
\end{split}
\]
If $f:A\to B$ is a morphism between two abelian varieties in $\AV_h(p)$, then for the dual morphism $f^{\vee}: B^{\vee} \to A^{\vee}$ we obtain
\begin{equation*}
\mathcal{G}(f^{\vee}) = \bar{\mathcal{G}(f)}.
\end{equation*}
\end{prop}
\begin{proof}
Recall that $\cF_w(f):\cF_w(B)\to\cF_w(A)$ sends $h \mapsto h \circ f$.
Since $\lambda_B$ is the basis element of~$\cF_w(B)$, we have that $\lambda_B = \lambda_B \circ i_B(\cG(\lambda_B))$, which implies that $\cG(\lambda_B) = i_B^{-1}(\mathrm{id}_B) = 1$.
Hence we have $i_A(\cG(\lambda_B))=\mathrm{id}_A$, because $i_A$ is a ring isomorphism.
In particular, we obtain
\[
\begin{split}
\cF_w(f)(\lambda_B) = \lambda_B \circ f & = \lambda_A\circ i_A(\cG(\lambda_B\circ f))\\
        & = \lambda_A\circ i_A(\cG(f)) \circ i_A(\cG(\lambda_B)) \\
        & = \lambda_A\circ i_A(\cG(f)).
\end{split}    
\]
This implies that (after tensoring with $\mathbb{Q}$) we have an expression
\begin{equation}\label{eq:Gf}
f = \lambda_B^{-1} \circ\lambda_A\circ i_A(\cG(f)).
\end{equation}
Starting from $\cF_w(f^\vee)$, by the same construction we obtain 
\begin{equation}\label{eq:Gfvee}
\begin{split}
f^{\vee} & = \lambda_{A^\vee}^{-1} \circ\lambda_{B^\vee}\circ i_{B^\vee}(\cG(f^\vee))\\
    & = \lambda_{A^\vee}^{-1} \circ i_{A_w}(\cG(f^\vee)) \circ\lambda_{B^\vee} \\
    & = i_{A^\vee}(\cG(f^\vee)) \circ \lambda_{A^\vee}^{-1} \circ \lambda_{B^\vee},
\end{split}    
\end{equation}
where the second and third equalities hold because of Remark \ref{rmk:leftSrightRw}.
Equations~\eqref{eq:Gf} and~\eqref{eq:Gfvee} yield
\begin{equation}\label{eq:GfGfvee}
\begin{split}
\cG(f) & = i_A^{-1} \circ \lambda_A^{-1} \circ \lambda_B \circ f;\\
\cG(f^\vee) & = i_{A^\vee}^{-1} \circ f^\vee \circ \lambda_{B^\vee}^{-1}\circ \lambda_{A^\vee}.
\end{split}    
\end{equation}
Observe that by Remark~\ref{rem:Llinearity}.2 we have
\begin{equation}\label{eq:Gfbar}
\begin{split}
\bar{\cG(f)} & = i_{A^\vee}^{-1} \circ \left( \lambda_A^{-1} \circ \lambda_B \circ f  \right)^\vee \\
    & = i_{A^\vee}^{-1} \circ f^\vee \circ \lambda_B^\vee \circ (\lambda_A^\vee)^{-1}.
\end{split}    
\end{equation}
Comparing~\eqref{eq:GfGfvee} and~\eqref{eq:Gfbar}, we see that $\bar{\cG(f)} = \cG(f^\vee)$ if and only if $ \lambda_{B^\vee}^{-1}\circ \lambda_{A^\vee} = \lambda_B^\vee \circ (\lambda_A^\vee)^{-1} $, or equivalently 
\[ \lambda_{A^\vee} \circ \lambda_A^{\vee} = \lambda_{B^\vee} \circ \lambda_B^\vee. \]
By our assumption both sides of this equality are equal to $\eta$.
\end{proof}

\begin{remark}\label{rem:choicelambda}
The upshot of the preceding results is as follows: 
\begin{enumerate}
    \item Corollaries~\ref{cor:JJbar} and~\ref{cor:Fwmod} show that we can modify the functor $\cG$ such that $\cG(A^{\vee}) \simeq \overline{\cG(A)}^t$. By modifying the basis elements if necessary, we can therefore moreover ensure that $\cG(A^{\vee}) = \overline{\cG(A)}^t$.
    \item The choice of basis elements $\lambda_{A_k} = \lambda_{A_k^{\vee}} \circ \lambda_k$ in Lemma~\ref{lem:GhomPsi} for a variety in the isogeny class which admits a canonical lifting, and  the choices of basis elements $\lambda_B \circ \lambda_B^{\vee} = \eta$ in Proposition~\ref{prop:dual} for any variety $B$ in the isogeny class can be made simultaneously. 
    
    That is, both $\Psi^{-1}_{\lambda_k}(\Hom_L(A_k,A_k^{\vee})) = \cG(\Hom_L(A_k,A_k^{\vee}))$ and $\cG(f^{\vee}) = \overline{\cG(f)}$ for any $f \in \Hom(A,B)$ hold.
\end{enumerate}
\textbf{For the remainder of the paper, we will assume to have made the choices of Remark~\ref{rem:choicelambda}.}
Choose an abelian variety $A/K$ with CM by an algebra $L$ and with (good) reduction to $A_k/k$, satisfying $i_k^{-1}(\End(A_k)) = i^{-1}(\End(A)) = S$ for a Gorenstein order $S$ such that $S = \overline{S}$ (as in Proposition~\ref{prop:redhomS}, cf.~Remark~\ref{rem:Gorenstein}). Then the following diagram summarizes the notation and choices made throughout this section, where $\mathrm{red}$ denotes the reduction map and where all arrows have been shown to be bijections: 
\[\xymatrix{
\Hom(A,A^\vee) \ar[d]_{\mathrm{red}} \ar@<2pt>[dr]^{\mathcal{H}} & \\
\Hom(A_k,A_k^\vee)\ar[d]^{\cG} & (\bar I^t : I) \ar@<2pt>[ul]^{\Psi_{\lambda}} \ar[d]_{\alpha} \\
(\cG(A_k) : \cG(A_k^\vee)) & (\bar I^t : I) \ar[ul]_{\Psi_{\lambda_k}} \ar@{=}[l]
}\]
\end{remark}

\section{Polarizations}\label{sec:pols}

In this section we prove the main result of this paper (Theorem~\ref{thm:main1}), which describes all polarizations of a given isogeny class in characteristic~$p$ when one variety in the isogeny class admits a CM lifting.

In Subsection~\ref{ssec:spreading}, we formally define polarizations and explain how to produce polarizations of varieties in a fixed isogeny class from one given polarization, reducing the problem to that of finding \emph{one} polarization in characteristic~$p$. The latter is carried out in Subsection~\ref{ssec:redpol}, where we exploit the connections to characteristic zero provided by CM-liftings and reduction discussed in Sections~\ref{sec:lifting} and~\ref{sec:cats}, together with the fact that polarizations in characteristic zero admit a very nice characterization (see Lemma~\ref{lem:Hpol}). 

\subsection{Spreading polarizations within isogeny classes}\label{ssec:spreading}

In this subsection we study isogenies between, and polarizations of, abelian varieties over an arbitrary field $K$. 

For an abelian variety $A$ over a field $K$, let $A^\vee$ be its dual, as before.
Since there exists a canonical isomorphism $A\simeq (A^{\vee})^{\vee}$, we will identify these below.

\begin{df}\label{def:symhom}
A homomorphism $\mu:A\to A^\vee$ is \emph{self-dual} if the dual homomorphism $\mu^{\vee}: A \to A^{\vee}$ satisfies $\mu=\mu^\vee$. 

Alternatively, for a line bundle $\cL$ on $A$, let $\varphi_{\cL} : A\to A^\vee$ be the homomorphism which on points is given by $x\mapsto [t^*_x\cL\otimes \cL^{-1}]$, where $t_x$ is the translation-by-$x$ map. 
Then $\mu$ is self-dual if there exist a finite field extension $K\subseteq K'$ and a line bundle $\cL$ on $A_{K'}$ such that $\mu_{K'}=\varphi_{\cL}$, cf.~\cite[Theorem 13.7]{polishchuk}.
\end{df}

\begin{df}\label{def:pol}
An isogeny $\mu:A\to A^\vee$ is a \emph{polarization} of $A$ if there exist a finite field extension $K \subseteq K'$ and an \emph{ample} line bundle $\cL$ on $A_{K'}$ such that
\begin{equation}\label{eq:defpol}
\mu_{K'}=\varphi_{\cL}. 
\end{equation}
In particular, polarizations are self-dual isogenies.
\end{df}

The following lemma will be used in the main theorem of this section (Theorem~\ref{thm:main1}).

\begin{lemma}\label{lem:isogsymm}
 Let $A$ and $B$ be abelian varieties over a field $K$ and let $f:A\to B$ and $\mu:B\to B^\vee$ be isogenies.
 Put $f^*\mu :=(f^\vee\circ\mu \circ f)$.
 Then $\mu$ is self-dual if and only if $f^*\mu$ is self-dual.
 Moreover, $\mu$ is a polarization of $B$ if and only if $f^*\mu$ is a polarization of $A$.
\end{lemma}

\begin{proof}
The first statement follows by direct computation, making use of the fact that $\mathrm{End}(B)$ is torsion-free.
A proof for the second statement can be found in \cite[Propositions 11.8 and 11.25(ii)]{MvdGE}. It uses the first statement, as well as the fact that $f^*$ preserves ampleness (for the forward implication) and effectiveness (for the reserve implication).
\end{proof}

\subsection{Characterizing polarizations}\label{ssec:redpol}

In this subsection, we revert to Notation~\ref{not3}: we let $K$ be a $p$-adic field with residue field $k$ and let an abelian variety $A/K$ (with good reduction) have reduction $A_k = A \otimes k$. Let $h$ be a squarefree Weil $p$-polynomial without real roots. 

We start with two useful results from the literature that characterize polarizations.
\begin{lemma}\label{lem:redpol}
Let $A/K$ and $A_k/k$ be as above.
Let $\mu_K: A \to A^{\vee}$ be an isogeny and consider its reduction $\mu_k : A_k \to A_k^{\vee}$.
Then $\mu_K: A\to A^\vee$ is a polarization if and only if $\mu_k: A_k \to A^\vee_k$ is a polarization.
\end{lemma}

\begin{proof}
See e.g.~\cite[p.~46 and Lemma~2.1.1.1]{chaiconradoort14}. 
\end{proof}

For the remainder of this subsection, we further assume that $A/K$ has CM by an algebra~$L$.

\begin{lemma}\label{lem:Hpol}
Let $A/K$ be as above with CM by an algebra $L$. Let $\mu_K: A\to A^\vee$ be an $L$-linear isogeny.
Put $\lambda=\mathcal{H}(\mu_K)\in~L$. Then $\mu_K$ is a polarization if and only if $\lambda$ is totally imaginary (that is, $\overline{\lambda} = -\lambda$)  and $\Phi$-positive (that is, $\Im(\vphi(\lambda))>0$ for every $\vphi\in\Phi$).
\end{lemma}
\begin{proof}
An isogeny $\mu_K$ is a polarization if and only if there is a field extension $K \subseteq K'$ and an ample line bundle over $K'$ that induces the extension $\mu_{K'}$, see for instance \cite[p. 126]{arithgeom86}.
Let $\mu_{\C}$ be the image of $\mu_K$ under the isomorphism from Corollary~\ref{cor:allhoms}. Then, in particular, $\mu_{\C}$ is induced by an ample line bundle if and only if $\mu_{K'}$ is. 

Put $I=\mathcal{H}(A)$. We follow the proof of \cite[Proposition 4.9]{Howe95}, noting that this part of the proof holds for any complex abelian variety with CM by $L$, 
and get a sesquilinear form $S:I \times I \to L$ with $S(s,t)=\overline{\lambda t} s$ (see also the proof of \cite[Theorem 5.4]{MarAbVar18}). 

It is then shown that $\mu_{\C}$ is a polarization if and only if $S$ is skew-symmetric and $-S(t,t)$, and hence $-\bar \lambda=\lambda$, are $\Phi$-positive.
\end{proof}

In order to be able to characterize polarizations over the finite field in the main theorem, we will need to get the best possible control over the element $\alpha$ in Proposition \ref{prop:redhomS}, which realizes the reduction map on the level of fractional ideals.

\begin{prop}\label{prop:betteralpha}
Let $A/K$ be an abelian variety with CM by an algebra $L$ that reduces to $A_k$ with $S\simeq\End(A)\simeq\End(A_k)$ and let $\alpha\in S^*$, as in Proposition \ref{prop:redhomS}.
Then $\alpha = \overline{\alpha}$; that is, $\alpha$ is totally real.
\end{prop}
\begin{proof}
Let $\mu:A\to A^\vee$ be a polarization of $A$ (of any degree).
By Lemma \ref{lem:Hpol} we have that 
$\mathcal{H}(\mu)= - \overline{\mathcal{H}(\mu)}$.
So by Proposition \ref{prop:redhomS} we obtain
\begin{equation}\label{eq:betteralpha:1}
    \cG(\mu_k)=\alpha\mathcal{H}(\mu) = -\alpha\overline{\mathcal{H}(\mu)}.
\end{equation}
By Lemma \ref{lem:redpol}, the reduction $\mu_k$ of $\mu$ is a polarization of $A_k$, and in particular it is self-dual, i.e. $\mu_k=\mu_k^{\vee}$ (via double duality for abelian varieties). From the proof of Proposition~\ref{prop:CSdual} we see that its induced action $T_p(\mu_k)$ on $T_p(A)$ is multiplication by $\cG(\mu_k)$. By \cite[Theorem 1.4.3.4]{chaiconradoort14}, using the categorical equivalence between $p$-divisible groups and Dieudonn\'e modules, self-duality of $\mu_k$ is equivalent to the statement that $T_p(\mu_k)=-T_p(\mu_k^{\vee})$ via double duality for $p$-divisible groups or Dieudonn\'e modules, which in this case is the identity. It follows that $\cG(\mu_k)=-\cG(\mu_k^\vee)$. By Proposition \ref{prop:dual}, we have
$\cG(\mu_k^\vee) = \overline{\cG(\mu_k)}$, which gives us
\begin{equation}\label{eq:betteralpha:2}
\cG(\mu_k)=-\cG(\mu_k^\vee)=-\overline{\cG(\mu_k)}=-\overline{\alpha \mathcal{H}(\mu)}.
\end{equation}
Comparing \eqref{eq:betteralpha:1} and \eqref{eq:betteralpha:2}, we deduce that $\alpha = \overline{\alpha}$.
\end{proof}

Using the characterizations in the previous two lemmas, we now proceed to the main result. To state it, we need one additional definition.

\begin{df}\label{def:S}
Let $\Phi$ be a CM-type for $L$. Define $\mathcal{S}_{\Phi}$ to be the set of orders $S$ in~$L$ which are Gorenstein and satisfy $\overline{S} = S$, and for which there exists an abelian variety $A_0 \in \AV_h(p)$ with CM-type $\Phi$ and $\End(A_0) = S$ that admits a canonical lifting $A$ to a $p$-adic field $K$.
\end{df}

\begin{prop}\label{prop:CCO_effective}
Assume that $\cO_L\in \mathcal{S}_{\Phi}$.
Let $S\subseteq \cO_L$ be an order such that the cardinality $N=\vert \cO_L/S \vert$ is coprime to $p$.
Then there exists an abelian variety $B_0 \in \AV_h(p)$, with $\End(B_0)=S$, that admits a canonical lifting to a $p$-adic field $K$.
In particular, if $S$ is also Gorenstein and satisfies $S=\overline{S}$ then $S\in \mathcal{S}_{\Phi}$.
\end{prop}
\begin{proof}
Since $\cO_L\in \mathcal{S}_{\Phi}$, there exists an abelian variety $A_0$ admitting a CM-lifting with $\End(A_0)=\cO_L$. Put $\cG(A_0)=I$.
Let $A'_0$ and $B_0$ be such that $\cG(A'_0)=\cO_L$ and $\cG(B_0)=S$.
Since the argument of \cite[Theorem 5.1]{Wat69} generalizes from the simple to the squarefree case, there is a separable isogeny $f_1:A_0\to A'_0$.
By assumption, $N$ is coprime to $p$ and $N\cO_L\subseteq S$, that is, $N\in (S:\cO_L)$.
It follows from \cite[Proposition 29.2]{CentelegheStix15} that $f_2:=\cG^{-1}(N): B_0\to A'_0$ is a separable isogeny.
By \cite[Theorem 5.2]{Wat69}, separable isogeny is an equivalence relation, so there exists an isogeny $f'_2: A'_0\to B_0$. In particular, $f'_2\circ f_1:A_0\to B_0$ is also a separable isogeny.
By Proposition \ref{prop:sepisoglifting}, we can then lift $B_0$ with its action by $\End(B_0)=S$ to~$K$.
\end{proof}

\begin{cor}\label{cor:Snonempty}\ 
\begin{enumerate}
    \item \label{cor:Snonempty:ord_almord} If $\AV_h(p)$ is an isogeny class of ordinary or almost-ordinary abelian varieties (in odd characteristic), then $R_w \in \mathcal{S}_{\Phi}$.
    \item If $(L,\Phi)$ satisfies the generalized residual reflex condition of Definition~\ref{def:RRC}, then we have $\mathcal{O}_L \in~\mathcal{S}_{\Phi}$.
    \item If $(L,\Phi)$ satisfies the generalized residual reflex condition of Definition~\ref{def:RRC}, then every overorder $S$ of $R_w$ such that $S$ is Gorenstein, $S=\overline{S}$ and $\vert \cO_L/ S \vert$ is coprime to $p$ is in $\mathcal{S}_{\Phi}$.
\end{enumerate}
\end{cor}

\begin{proof}\
\begin{enumerate}
    \item Choose $A_0 \in \AV_h(p)$ with $\mathrm{End}(A_0) = R_w$. By Propositions~\ref{prop:ord_lift} and \ref{prop:almord_lift}, this variety admits a canonical lifting $A_K$, so $\mathrm{End}(A_0) = \mathrm{End}(A_K)$. Moreover, $R_w$ is Gorenstein and satisfies $\overline{R}_w = R_w$ by \cite[Theorem 11]{CentelegheStix15}. 
    \item It follows from Corollary~\ref{cor:CMlift} that there exists an $A_0 \in \AV_h(p)$ such that $\mathcal{O}_L \subseteq \mathrm{End}(A_0)$, which has a CM-lifting $A_K$ satisfying $\mathcal{O}_L \subseteq \mathrm{End}(A_K)$. Hence, $\mathcal{O}_L = \mathrm{End}(A_0) = \mathrm{End}(A_K)$ by maximality. In addition, $\mathcal{O}_L$ is Gorenstein and satisfies $\overline{\mathcal{O}}_L = \mathcal{O}_L$.
    \item This follows directly by Part 2. and Proposition \ref{prop:CCO_effective}.
\end{enumerate}
\end{proof}

\begin{thm}\label{thm:main1}
For any $S \in \mathcal{S}_{\Phi}$, there exists a totally real unit $\alpha \in S^*$ (provided by~Proposition~\ref{prop:betteralpha}) such that for any abelian variety $B_0 \in \AV_h(p)$, and any isogeny $\mu: B_0 \to B_0^{\vee}$, the following are equivalent:
\begin{enumerate}
\item The isogeny $\mu$ is a polarization of $B_0$;
\item The element $\alpha^{-1} \cG(\mu) \in L$ is totally imaginary and $\Phi$-positive.
\end{enumerate}
Moreover, we have $\deg \mu = \# \left( \cG(B_0) / \cG(\mu) \cG(B_0^\vee) \right)$.
\end{thm}
\begin{proof}
Let $A_0 \in \AV_h(p)$ admitting a canonical lifting $A/K$ with $\End(A_k)=S$.
Fix an isogeny $f: A_0 \to B_0$.
The previous results yield the following commutative diagram, where $\mathrm{red}$ denotes the reduction map:
\[\xymatrix{
                                 & \Hom(A,A^\vee) \ar[d]_{\mathrm{red}} \ar[dr]^{\mathcal{H}} & \\
 \Hom(B_0,B_0^\vee) \ar[r]^{f^*}\ar[d]^{\cG} & \Hom(A_0,A_0^\vee)\ar[d]^{\cG} & (\bar I^t : I) \ar[d]_{\alpha} \\
 (\cG(B_0) : \cG(B_0^\vee))\ar[r]^{\cG(f^*)} & (\cG(A_0) : \cG(A_0^\vee)) & (\bar I^t : I)
 \ar@{=}[l]
}\]

Let $\mu:B_0\to B_0^\vee$ be an isogeny. 
By Lemma~\ref{lem:isogsymm}, the isogeny $\mu$ is a polarization of $B_0$ if and only if $f^*\mu$ is a polarization of $A_0$, which by Lemma \ref{lem:redpol} happens if and only if $\mathrm{red}^{-1}(f^*\mu)$ is a polarization of $A$. By Lemma \ref{lem:Hpol}, this is the case if and only if $$\mathcal{H}\left(\mathrm{red}^{-1}(f^*\mu)\right)=\Psi_\lambda^{-1}(\mathrm{red}^{-1}(f^*\mu))$$ 
is totally imaginary and $\Phi$-positive.
By commutativity of the diagram we see that
\[ \mathcal{H}\left(\mathrm{red}^{-1}(f^*\mu)\right) =  
\frac{1}{\alpha}\left(\cG(f^*\mu)\right). \]
By Proposition \ref{prop:dual} we have that $\cG(f^*\mu) = \cG(f)\overline{\cG(f)}\cG(\mu)$, which implies that $\alpha^{-1}(f^*\cG(\mu))$ is totally imaginary and $\Phi$-positive if and only if $\alpha^{-1}\cG(\mu)$ is so.
The degree statement follows from \cite[Proposition 29.2]{CentelegheStix15}.
\end{proof}

\begin{prop}\label{prop:isompol}
Let $B_0 \in \AV_h(p)$; let $T \subseteq L$ denote the embedded order $\mathrm{End}(B_0)$.
Let $\mu, \mu'$ be two polarizations of $B_0$. 
Then $(B_0,\mu)$ and $(B_0,\mu')$ are isomorphic (as polarized abelian varieties) if and only if there exists $v \in T^*$ such that $\cG(\mu) = v\overline{v}\cG(\mu')$.
In particular, $\mathrm{Aut}((B_0, \mu)) = \mathrm{Tors}(T^*)$.
\end{prop}

\begin{proof}
The existence of an isomorphism $f: (B_0,\mu) \to (B_0,\mu')$ is equivalent to the commutativity of the following diagram:
\[
\begin{tikzcd}[swap,bend angle=45]
 B_0 \rar{\mu} \dar{f} & B_0^{\vee} \\ 
 B_0 \rar{\mu'} & B_0^{\vee} \uar{f^{\vee}}
\end{tikzcd}
\]
If $\cG(f) = v \in T$ , then $\cG(f^{\vee}) = \overline{v} \in T$ by Proposition~\ref{prop:dual}, proving the first statement.

When $\cG(\mu) = \cG(\mu')$, the diagram shows that $v \overline{v} = 1$.
Since $L$ is a CM-algebra, this is the case if and only if $v \in \mathrm{Tors}(T^*)$:
We use that $\mathrm{Tors}(T^*) = \mathrm{Tors}(L^*) \cap T^*$ and study $\mathrm{Tors}(L^*)$. 
By~\cite[Chapter I, Proposition 7.1]{Neukirch99}, we have 
\[
\mathrm{Tors}(L^*) = \{ v \in \mathcal{O}_L^* : \vert \phi(v) \vert = 1 \text{ for all } \phi \in \Hom(L,\mathbb{C}) \}.
\]
Now $\vert \phi(v) \vert = 1$ if and only if $1 = \vert \phi(v) \vert^2 = \phi(v) \overline{\phi(v)} = \phi(v) \phi(\overline{v}) = \phi(v \overline{v})$; for the third equality we use that $L$ is a CM-algebra.
And $\phi(v \overline{v}) = 1$ for all $\phi$ is equivalent to $v \overline{v} = 1$.
\end{proof}

\section{Algorithmic implementations}\label{sec:algo}
In this section we describe algorithms we developed to apply the theoretical results obtained in the previous sections.
More precisely, in Subsection \ref{ssec:RRCcomp} we explain how to verify that in a suitable isogeny class there exists an abelian variety admitting a canonical lifting and, given this, in Subsection \ref{ssec:polcomp} we describe how to effectively compute the principally polarized abelian varieties. Finally, in Subsection \ref{ssec:examples} we give some illustrative examples.

The implementation of the algorithms described and the code to reproduce the examples can be found at \url{https://github.com/stmar89/PolsAbVarFpCanLift}.
As mentioned in the introduction, we computed the principal polarizations for all squarefree isog-eny classes of abelian varieties of dimensions two and three over the finite fields $\F_3$, $\F_5$ and $\F_7$, and of dimension four over $\F_2$ and $\F_3$. These computations show that the generalized residual reflex condition produces a lifting in the vast majority of cases and that in most cases we can ignore the contribution given by $\alpha$.
Tables summarizing these results can be found in the file
\href{https://github.com/stmar89/PolsAbVarFpCanLift/blob/main/additional_examples.pdf}{{\tt additional\_examples.pdf}} on GitHub.

\subsection{The generalized residual reflex condition}\label{ssec:RRCcomp}
Recall from Definition~\ref{def:padicCMtype}.\ref{def:padicCMtype:cmtype} that a CM-type $\Phi$ for us can be $\C$-valued or $\overline{\Q}_p$-valued. Since these fields are abstractly isomorphic, one can theoretically pass from one to the other; for algorithmic implementations however, making an identification $\overline{\Q}_p \simeq \C$ is not straightforward.

Note that Section~\ref{sec:lifting} requires the usage of $p$-adic CM-types, since the generalized residual reflex condition of Definition~\ref{def:RRC} 
is a $p$-adic condition, while in Section~\ref{sec:pols} we need to work with complex CM-types in order to evaluate the conditions in Theorem~\ref{thm:main1}.2.

In this section, we will explain how to implement an algorithm that determines, given a squarefree characteristic polynomial of Frobenius $h$ and any $\C$-valued CM-type $\Phi$ for $L = \Q[x]/(h)$, an associated $\bar{\Q}_p$-valued CM-type on which we can check whether the generalized residual reflex condition is satisfied.

First we construct splitting fields $M/\Q$ and $N/\Q_p$ for $h$ and fix embeddings $\varphi_0\colon M \hookrightarrow \mathbb{C}$ and $N \hookrightarrow \overline{\mathbb{Q}}_p$.
Observe that there exists an embedding $j': M \to N$, 
which extends to an isomorphism $j: \mathbb{C} \to \overline{\mathbb{Q}}_p$.
Factor the polynomial $h=h_1\cdot \ldots \cdot h_t$ into irreducible factors and write $(L,\Phi)=(L_1\times\ldots\times L_t,\Phi_1\times\ldots\times \Phi_t)$, where $L_i=\Q[x]/(h_i) = \Q(\pi_i)$ and $\Phi_i$ is the induced CM-type on $L_i$.
Denote by $\Pi_i = \{\pi_{i,1},\ldots,\pi_{i,{2g_i}}\}$ the roots of $h_i$ in $M$, where $2g_i = \deg(h_i)$.
The embeddings $j'$ and $\varphi_0$ allow us to identify $\Phi_i$ with a subset $\Pi_{\Phi_i}$ of $\Pi_i$, say, after possibly relabelling, $\Pi_{\Phi_i} = \{\pi_{i,1},\ldots,\pi_{i,g_i} \}$.
It follows from Definition~\ref{def:padicCMtype} that the reflex field associated to $\Phi_i$ is equal to
\begin{equation*}
E_i=\Q\left( \pi_{i,1}^n + \ldots + \pi_{i,g_i}^n : 0\leq n \leq g_i-1 \right),
\end{equation*}
and the compositum $E=E_1\cdot\ldots \cdot E_t$ equals
\begin{equation}\label{eq:E}
E=\Q\left( \pi_{i,1}^{n_i} + \ldots + \pi_{i,g_i}^{n_i} :0\leq {n_i} \leq g_i-1,\  0\leq i \leq t \right) \subseteq M. 
\end{equation}
Note that $j'$ induces a $p$-adic valuation $\nu$ on $E$.
The image $j'(E)$ does not contain $\mathbb{Q}_p$, so we consider its completion $\widehat{j'(E)}$ 
with respect to $\nu$.
By construction, the residue field of $E$ at $\nu$ is isomorphic to that of ${\widehat{j'(E)}}$. 
The maps fit together in the following commutative diagram:
\begin{equation}\label{eq:diagCQp}
\begin{tikzcd}[swap,bend angle=45]
\mathbb{C} \ar[rr,"j"] & & \overline{\mathbb{Q}}_p \\
M \uar{}{\varphi_0} \ar[rr,"j'"] & &  N \uar{} \\
E \uar{} \rar{} & j'(E) \rar{} & {\widehat{j'(E)}} \uar{} \\
\mathbb{Q} \uar{} \ar[rr] & & \mathbb{Q}_p \uar{}
\end{tikzcd}
\end{equation}

To check whether the residue field of $E$ at $\nu$ is a subfield of $\mathbb{F}_q$, as needed for Condition~\ref{def:RRC_item_refl} of Definition~\ref{def:RRC}, we can work inside $\mathbb{C}$ (i.e., using the left vertical maps in Diagram~\eqref{eq:diagCQp}). This would require the following expensive computations:
\begin{enumerate}
    \item An explicit factorization of $p\in \mathbb{Q}$ in $M$ in order to compute the prime ideal $\mathfrak{p} \subseteq M$ corresponding to $\nu$;
    \item A presentation of $E$ as a subfield of $M$, as in \eqref{eq:E};
    \item An integral basis of $E$ in order to calculate $\mathcal{O}_{E}/(\mathfrak{p}\cap E)$.
 \end{enumerate}
The upshot of Diagram~\eqref{eq:diagCQp} however is that, given the generators of $E$ as in~\eqref{eq:E}, we can instead work inside $\overline{\mathbb{Q}}_p$ (i.e., using the right vertical maps in Diagram~\eqref{eq:diagCQp}).
Computations 1.-3. become faster in $p$-adic fields.

To summarize, using this $p$-adic approach, our algorithm for checking Condition \ref{def:RRC_item_refl}~of Definition~\ref{def:RRC} works as follows:
\begin{enumerate}
\item Compute a presentation $M = \mathbb{Q}[x]/(f)$;
\item Construct the embeddings $j': M \to N$ and $\varphi_0: M \to \mathbb{C}$ by computing a root of $f$ in $N$ and $\mathbb{C}$, respectively; 
\item \label{alg:rrc_refl_3} Construct ${\widehat{j'(E)}}$ as the extension of $\mathbb{Q}_p$ generated by the images of the generators of $E$ under $j'$;
\item \label{alg:rrc_refl_4} Compute the residue field of ${\widehat{j'(E)}}$;
\item \label{alg:rrc_refl_5} Check whether this residue field is contained in $\mathbb{F}_q$.
\end{enumerate}
Observe that if the residue field of $N$ is contained in $\mathbb{F}_q$, then the same will be true for the residue field of any subfield of $N$. 
In particular, in this situation, we can skip Steps~\ref{alg:rrc_refl_3},~\ref{alg:rrc_refl_4}, and~\ref{alg:rrc_refl_5} of the algorithm above.

Next, we check Condition \ref{def:RRC_item_st} of Definition~\ref{def:RRC}.
Recall that we need to have that the Shimura-Taniyama formula
\begin{equation}\label{eq:ST2}
\dfrac{\ord_\nu(\pi)}{\ord_\nu(q)}=\dfrac{\#\set{ \vphi \in \Phi \text{ s.t.~} \vphi \text{ induces } \nu }}{[L_\nu:\Q_p]}
\end{equation} 
holds for every place $\nu$ of $L$ above $p$.
There exists a bijection between the primes $\mathfrak{p}$ of $L$ above $p$ and the irreducible factors of $h$ over $\mathbb{Q}_p$, cf.~\cite[II.8.2]{Neukirch99}.
Explicitely, we associate to every prime $\mathfrak{p}$ of L above $p$ the minimal polynomial $h_{\mathfrak{p}}$ over $\mathbb{Q}_p$ of the image of $\pi$ in the completion $L_{\mathfrak{p}}$.
Denote by $\Pi$ the union of the sets $\Pi_i$, and by $\Pi_{\Phi}$ the union of the sets $\Pi_{\Phi_i}$.
Checking whether~\eqref{eq:ST2} holds for all $\nu$ above $p$ is now equivalent to checking whether
\begin{equation}\label{eq:ST3}
\dfrac{\ord_{\mathfrak{p}}(\pi)}{\ord_{\mathfrak{p}}(q)}=\dfrac{\#\set{ \pi_i \in \Pi_{\Phi} \text{ s.t.~} h_{\mathfrak{p}}(j'(\pi_i)) = 0 }}{[L_{\mathfrak{p}} :\Q_p]},
\end{equation}
holds for primes $\mathfrak{p}$ of $L$ above $p$.
Note that 
\[
[L_{\mathfrak{p}} :\Q_p] = \#\set{ \pi_i \in \Pi \text{ s.t.~} h_{\mathfrak{p}}(j'(\pi_i)) = 0 } = \deg(h_{\mathfrak{p}}).
\]

Hence, our algorithm for checking Condition \ref{def:RRC_item_st}~of~Definition~\ref{def:RRC} checks whether Equation~\eqref{eq:ST3} is satisfied for every prime $\mathfrak{p}$ of $L$ above $p$.

\begin{remark}\label{rmk:CMtotimg}
A $\C$-valued CM-type $\Phi$ for $L = \mathbb{Q}[\pi]$ is determined by a totally imaginary element $b\in L$.
Indeed, using $\varphi_0$ as in Diagram \eqref{eq:diagCQp}, we can build a bijection between homomorphisms $\vphi_i\colon L \to \mathbb{C}$ and the $2g$ roots of $h$ in $M$ through $\vphi_i(\pi) = \pi_{i}$. 
Then a CM-type $\Phi$ determines a totally imaginary element $b=\sum_k b_k \pi^k$ with $b_k\in \Q$ via the following rule:
\begin{align*}
\vphi_i \in \Phi & \iff  \Im(\vphi_i(b))>0 \\ 
& \iff \sum_{k=0}^{2g-1}b_k\Im(\vphi_i(\pi^k))>0 \\
& \iff \sum_{k=0}^{2g-1}b_k\Im(\pi_i^k)>0.
\end{align*}
This element $b$ is unique up to a totally positive element of $L$.
Conversely, every totally imaginary element determines a unique CM-type of $L$.
\end{remark}

\subsection{Explicit computations of principal polarizations}\label{ssec:polcomp}
We will use the results of the previous sections to compute all principal polarizations of an arbitrary abelian variety in a fixed isogeny class determined by a squarefree Weil polynomial $h$. 
Set $L = \mathbb{Q}[x]/(h)$ and denote by $L_\R$ the unique totally real subalgebra of $L$. Moreover let $L^{+}$ denote the of totally positive elements of $L_\R$.
Let $b$ be a totally imaginary element of~$L$ and let $\Phi_b$ be the associated $\C$-valued CM-type of $L$, as in Remark \ref{rmk:CMtotimg}.
Suppose that the set~$\mathcal{S}_{\Phi_b}$, cf.~Definition~\ref{def:S}, is non-empty.
Fix $S \in \mathcal{S}_{\Phi_b}$ and let $\alpha \in S^*$ be the totally real element appearing in Proposition~\ref{prop:betteralpha}.
Define $S^*_\R=S^*\cap L_\R$ to be the subgroup of $S^*$ consisting of totally real elements.
It follows from Remark \ref{rmk:CMtotimg} that for $u\in S^*_\R$, the association 
\[(u,\Phi_b)\mapsto \Phi_{ub}\]
induces a free action of the finite group
\[ G_S := \frac{S^*_\R}{S^* \cap L^+} \]
on the set of CM-types of $L$.
Denote the orbit of $\Phi_b$ by
\[ \Phi_b\cdot G_S=\set{ \Phi_{ub} : u \in S^*_\R } .\]

Now let $B_0$ be an arbitrary abelian variety in $\AV_h(p)$; put $\cG(B_0)=I$ and $T=(I:I)$.
Denote by $\mathcal{T}$ a transversal of the finite group $T^*/\langle v\overline{v} : v \in T^* \rangle$.
A necessary condition for $B_0$ being principally polarized is the existence of an isomorphism $B_0\simeq B_0^\vee$, or equivalently, the existence of an element $i_0\in L^*$ such that $i_0 \overline{I}^t = I$.
Assume now that such an $i_0$ exists. 
For an arbitrary $\C$-valued CM-type $\Phi$ of $L$, we define
\begin{equation}\label{eq:Palpha}
\Palpha{\Phi}{I}:=\{ i_0  u  : u \in \mathcal{T} \text{ such that } \alpha^{-1} i_0 u \text{ is totally imaginary and } \Phi\text{-positive } \}
\end{equation}
and
\begin{equation}\label{eq:Pone}
\Pone{\Phi}{I}:=\{ i_0  u  : u \in \mathcal{T} \text{ such that } i_0 u \text{ is totally imaginary and } \Phi\text{-positive} \}.
\end{equation}
If $B_0$ is not isomorphic to $B_0^\vee$, that is, if $i_0$ as above does not exist, then we formally set $\Palpha{\Phi}{I}=\Pone{\Phi}{I}=\emptyset$, for all $\Phi$.
Observe that if we replace $\mathcal{T}$ with another transversal in \eqref{eq:Palpha} (resp.~\eqref{eq:Pone}), we will obtain a set which is in bijection with $\Palpha{\Phi}{I}$ (resp.~$\Pone{\Phi}{I}$).
The same holds true if we replace the element $i_0$ with another element $i'_0$ satisfying $i'_0 \overline{I}^t = I$, since we would then have $i_0=ui'_0$ for some $u\in T^*$.

\begin{prop}\label{prop:comppol}
For any choice of a transversal $\mathcal{T}$,
the set~$\Palpha{\Phi}{I}$ is the image under $\cG$ of a complete set of representatives of principal polarizations of $B_0$ up to isomorphism.
Moreover, if $\Palpha{\Phi}{I}$ is non-empty, then it contains precisely $\vert T^*_\R/\langle v\overline{v} : v \in T^* \rangle \vert $ elements, where $T^*_\R = T^* \cap L_\R$.
\end{prop}
\begin{proof}
By Theorem \ref{thm:main1} and Proposition~\ref{prop:isompol}, the elements of $\Palpha{\Phi}{I}$ are images under $\cG$ of pairwise non-isomorphic principal polarization of $B_0$.
Conversely, let $\mu$ be a principal polarization of $B_0$ and put $\lambda = \cG(\mu)$.
Then $\lambda i_0^{-1} = u$ is in $T^*$.
There exist a unique element $u_0 \in \mathcal{T}$ and an element $v \in T^*$ such that $u_0=uv\overline{v}$.
Now $\lambda'=v\overline{v}\lambda = i_0 u_0 \in \Palpha{\Phi}{I}$ represents a principal polarization of $B_0$, which is isomorphic to $\lambda$ by Proposition~\ref{prop:isompol}.
The last statement follows from the fact that if $\lambda$ and $\lambda'$ are in $\Palpha{\Phi}{I}$ then $\lambda/\lambda' \in T^*_\R$.
\end{proof}

The obstruction that prevents Proposition \ref{prop:comppol} from being effective is that the set $\Palpha{\Phi}{I}$ depends on the element $\alpha$, which, a priori, can be any totally real element of $S^*$.
In the rest of this section we will explain when and how we can compute a set in bijection with $\Palpha{\Phi}{I}$ without any further knowledge about $\alpha$.

\begin{lemma}\label{lem:PtoP'alpha}
Let $\Phi_b$ be a CM-type defined by a totally imaginary element $b\in L$. 
Then
\[ \Palpha{\Phi_b}{I} = \Pone{\Phi_{\alpha b}}{I}. \]
\end{lemma}
\begin{proof}
Let $i_0u \in \Palpha{\Phi_b}{I}$.
Since the element $\alpha$ is totally real, the element $\alpha^{-1} i_0 u$ is totally imaginary if and only if $i_0 u$ is. Moreover,
\[ \alpha^{-1} i_0 u \text{ is } \Phi_b\text{-positive } \iff \frac{\alpha^{-1} i_0u}{b}\in L^+ \iff  i_0 u \text{ is } \Phi_{\alpha b}\text{-positive }, \]
which concludes the proof.
\end{proof}

\begin{lemma} \label{lem:sameoutputs}
Let $\Phi_b$ be a CM-type of $L$ and let $c\in S^*_\R \cap T^*_\R$.
Put $d:=cb$, so that 
$\Phi_d\in \Phi_b G_S$.
Then there exists a bijection
\[ \Pone{\Phi_{b}}{I} \longleftrightarrow \Pone{\Phi_{d}}{I}. \]
\end{lemma}
\begin{proof}
Pick $i_0u \in \Pone{\Phi_{b}}{I}$. Observe that
\[ i_0 u \text{ is } \Phi_b\text{-positive } \iff \frac{i_0u}{b}=\frac{i_0uc}{d}\in L^+ \iff  i_0 uc \text{ is } \Phi_{d}\text{-positive }. \]
Hence, we have
\begin{equation}
\label{lem:sameoutputs:eq:2}
 \Pone{\Phi_{d}}{I} = \set{ i_0u : u \in c\mathcal{T} \text{ such that } i_0u \text{ is totally imaginary and } \Phi_d\text{-positive} }.     
\end{equation}
Since $c\in T^*$ by assumption, the set $c\mathcal{T}$ is another transversal of $T^*/\langle v \overline{v} : v \in T^* \rangle$.
Hence, we have a bijection 
\[ \set{ i_0u : u \in c\mathcal{T} \text{ such that } i_0u \text{ is totally imaginary and } \Phi_d\text{-positive} } \longleftrightarrow \Pone{\Phi_{b}}{I}, \]
which together with Equation \eqref{lem:sameoutputs:eq:2} yields the desired bijection. 
\end{proof}

\begin{thm}
\label{thm:eff}
Let~$\Phi_b$ be a CM-type for~$L$ for which~$\mathcal{S}_{\Phi_b}$ (cf.~Definition \ref{def:S}) is non-empty.
Let~$S$ be an order in~$\mathcal{S}_{\Phi_b}$ with group of totally real units~$S^*_\R$.
Let~$\alpha\in S^*_\R$ be the element representing the reduction map, cf.~Proposition \ref{prop:betteralpha}.
Let~$L^+$ be the set of totally positive elements of $L$.
Write~$G_S$ for the quotient~$S^*_\R/(S^* \cap L^+)$ which acts on the set of CM-types for $L$ and let~$\Phi_b\cdot G_S$ be the orbit of $\Phi_b$. Let $I$ be a fractional ideal with multiplicator ring $T=(I:I)$.
Denote a transversal of the quotient $T^*/\langle v \overline{v} : v \in T^* \rangle$ by $\mathcal{T}$.
Finally, for a CM-type $\Phi$ for $L$, let~$\Palpha{\Phi}{I}$ and~$\Pone{\Phi}{I}$ be as in Equations~\eqref{eq:Palpha} and~\eqref{eq:Pone}.

Consider the following statements:
\begin{enumerate}
    \item \label{thm:eff:1} There is an inclusion $S^*_\R \subseteq T^*_\R$.
    \item \label{thm:eff:2} For every $\xi \in S^*_\R$ we have $\xi L^+ \cap \mathcal{T} \neq \emptyset$.
    \item \label{thm:eff:4} There exists a bijection $\Palpha{\Phi_b}{I} \leftrightarrow \Pone{\Phi_b}{I}$.
    \item \label{thm:eff:5} For every $\Phi_d \in \Phi_b G_S$ there exists a bijection $\Pone{\Phi_b}{I} \leftrightarrow \Pone{\Phi_d}{I}$.
    \item \label{thm:eff:6} For every $\Phi_d \in \Phi_b G_S$ there exists a bijection $\Palpha{\Phi_b}{I} \leftrightarrow \Pone{\Phi_d}{I}$.
\end{enumerate}
We have the following implications:
\[
\xymatrix{
\ref{thm:eff:1} \ar@{=>}[r]\ar@{=>}[dr] & \ref{thm:eff:2} \ar@{=>}[r] & \ref{thm:eff:4} \\
    & \ref{thm:eff:5} \ar@{=>}[r] & \ref{thm:eff:6} \ar@{=>}[u]
}
\]
\end{thm}
\begin{proof}
The implications (\ref{thm:eff:1} $\Rightarrow$ \ref{thm:eff:2}) and  (\ref{thm:eff:6} $\Rightarrow$ \ref{thm:eff:4}) are clear.
For (\ref{thm:eff:2}$\Rightarrow$ \ref{thm:eff:4}), observe that $\left( \alpha^{-1} \cdot L^{+} \right) \cap \mathcal{T} \neq \emptyset$ since $\alpha \in S^*_\R$.
Hence there exist $u' \in \mathcal{T}$ and $r \in L^+$ such that $\alpha^{-1} = u'/r$. Pick $i_0u\in \Palpha{\Phi_b}{I}$.
Then $\alpha^{-1}i_0u$ is totally imaginary and $\Phi_b$-positive, which is equivalent to saying that $i_0uu'$ is so. Since $u' \mathcal{T}$ is another transversal of $T^*/\langle v \overline{v} : v \in T^* \rangle$, we get the required bijection.
For the implication (\ref{thm:eff:1} $\Rightarrow$ \ref{thm:eff:5}), let $\Phi_d\in \Phi_b G_S$. That is, there exists $c\in S^*_\R$ such that $d=cb$.
By assumption we also have that $c \in T^*_\R$, so we conclude by Lemma \ref{lem:sameoutputs}.
Finally, in order to prove (\ref{thm:eff:5} $\Rightarrow$ \ref{thm:eff:6}), observe that by definition $\Phi_{\alpha b}\in \Phi_bG_S$. That is, by assumption $\Pone{\Phi_d}{I}$ is in bijection with $\Pone{\Phi_{\alpha b}}{I}$ for any $\Phi_d \in \Phi_b G_S$. By Lemma~\ref{lem:PtoP'alpha} we have $\Pone{\Phi_{\alpha b}}{I} = \Palpha{\Phi_b}{I}$.
\end{proof}

\begin{remark}
To check whether the intersection $\xi L^{+} \cap \mathcal{T}$ in Theorem \ref{thm:eff}.\ref{thm:eff:2} is non-empty for every $\xi \in S_\R^*$, it is sufficient to check this for any generator $\xi$ of~$S_\R^*$.
This is a finite process; a straightforward implementation searches, for each generator $\xi$ of $S_\R^*$, whether there exists an element $u\in \mathcal{T}$ such that $u/\xi$ is totally positive.
\end{remark}

\begin{remark} \label{rmk:comp_pols}
Conditions \ref{thm:eff:1}, \ref{thm:eff:2}, and \ref{thm:eff:5} of Theorem~\ref{thm:eff} can be checked algorithmically on a computer.
If they hold, they respectively imply Conditions~\ref{thm:eff:4} and~\ref{thm:eff:6}, which do not involve the element $\alpha$.
In this way, Theorem \ref{thm:eff} yields concrete conditions under which we can compute the set $\Palpha{\Phi_b}{I}$ in terms of sets not involving the element $\alpha$. This yields an effective version of Proposition \ref{prop:comppol}.
\end{remark}

\begin{remark}\label{rem:polisogclass}
If Condition \ref{thm:eff:1} resp.~Condition \ref{thm:eff:2} of Theorem~\ref{thm:eff} holds when $T=R_w$, then it holds for every order in $L$. That is, we can compute the principal polarizations for the whole isogeny class. This is the case when the isogeny class is ordinary or almost ordinary (in odd characteristic), since in these cases we can take $S=R_w$ by Corollary \ref{cor:Snonempty}.\ref{cor:Snonempty:ord_almord}.
\end{remark}

Theorem \ref{thm:eff} gives us conditions under which we can determine the principal polarizations of a single abelian variety.
Even if these conditions are not met, we might still be able to get information about the total number of isomorphism classes of principally polarized abelian varieties in the isogeny class  for a fixed endomorphism ring.
\begin{thm}\label{thm:sameoutput}
Assume that there are $r$ isomorphism classes of abelian varieties in $\AV_h(p)$ with endomorphism ring $T$, represented by $B_{0,1},\ldots,B_{0,r}$, say. Put $I_i=\mathcal{G}(B_{0,i})$ for $i=1,\ldots,r$. Consider the following statements:
\begin{enumerate}
    \item \label{thm:sameoutput:1} There exists a non-negative integer $N$ such that for every $\Phi'$ in $\Phi_bG_S$ we have
\[
    \vert \Pone{\Phi'}{I_1} \vert + \ldots + \vert \Pone{\Phi'}{I_r} \vert = N.
\]
    \item \label{thm:sameoutput:2} There exist non-negative integers $N_1,\ldots,N_r$ such that for every
     $\Phi'$ in $\Phi_bG_S$ the sequence
\[
    \left( \vert \Pone{\Phi'}{I_1} \vert , \ldots , \vert \Pone{\Phi'}{I_r} \vert \right)
\]
is a permutation of $(N_1,\ldots,N_r)$.
\end{enumerate}
Then Statement~\ref{thm:sameoutput:1} holds if and only if Statement~\ref{thm:sameoutput:2} holds.
If this is the case, then there exists a permutation $\sigma$ of $\{1,\ldots,r\}$ such that for each $i = 1, \ldots, r$ the abelian variety $B_{0,i}$ has $C_{\sigma(i)}$ non-isomorphic principal polarizations.
Moreover, for each $i = 1, \ldots, r$ the integer $N_i$ is either $0$ or 
$\vert T_\R^*/\langle v\bar v : v \in T^*\rangle \vert $.
\end{thm}

\begin{proof}
Observe that Statement~\ref{thm:sameoutput:2} trivially implies Statement~\ref{thm:sameoutput:1}.
For the converse, note that by definition, for any~$\Phi'$, the set $\Pone{\Phi'}{I_i}$ has cardinality equal to either $0$ or $N_0:=\vert T_\R^*/\langle v\bar v : v \in T^*\rangle \vert $.
By Proposition \ref{prop:comppol} and Lemma \ref{lem:PtoP'alpha}, there are exactly $r_0$ abelian varieties among the $B_{0,i}$ admitting exactly $N_0$ non-isomorphic principal polarizations and $r-r_0$ which are not principally polarized.
So if Statement~\ref{thm:sameoutput:1} holds, then $N=r_0N_0$, and hence Statement~\ref{thm:sameoutput:2} holds for the sequence on integers defined by 
\[
N_i = \begin{cases}
N_0 & \text{ for } i = 1, \ldots, r_0, \\
0 & \text{ for } i = r_0 + 1, \ldots, r.
\end{cases}
\]
\end{proof}

\subsection{Examples}\label{ssec:examples}
In the following example we compute the principal polarizations of all abelian varieties in a fixed non-simple almost-ordinary isogeny class.
This illustrates the use of Theorem~\ref{thm:OS} (rather than Proposition~\ref{prop:almord_lift} together with Theorem~\ref{thm:main1}, since the finite field in the example is not a prime field.). 

\begin{example} \label{ex:nonord2}
Consider the Weil polynomial $h=(x^2 + 9)(x^4 - 6x^3 + 19x^2 - 54x + 81)$. It defines an isogeny class $\AV_h(9)$ 
of almost-ordinary threefolds over $\F_9$, which split up to isogeny into a supersingular elliptic curve and an ordinary surface.
See the LMFDB \cite{lmfdb} label \href{https://www.lmfdb.org/Variety/Abelian/Fq/3/9/ag_bc_aee}{{3.9.ag\_bc\_aee}}.
One verifies that the endomorphism algebra~$L$ of the isogeny class is commutative.
We compute that there are two overorders of~$R_w$ which are maximal at the supersingular part of $L$: the maximal order $\cO_L$ and an order $S$ of index $4$ in $\cO_L$.
The algebra $L$ has an inert supersingular part,
hence $\AV_h(9)$ is partitioned into two full subcategories $\AV_{h,\Phi_1}(9)$ and $\AV_{h,\Phi_2}(9)$, where $\Phi_1$ and $\Phi_2$ are the two CM-types of $L$ satisfying the Shimura-Taniyama formula, cf.~Definition \ref{def:RRC}.\ref{def:RRC_item_st}.
By \cite[Proposition 3.3]{OswalShankar19}, cf.~Theorem \ref{thm:OS} for the non simple case, for each $i=1,2$, there is a bijection between the isomorphism classes in $\AV_{h,\Phi_i}(9)$ and the isomorphism classes of fractional $S$-ideals.
Since $S$ is Gorenstein, we see that such ideal classes are in 
\[ \Pic(S) \sqcup \Pic(\cO_L) \simeq \Z/2\Z \sqcup \{1\}.\]
Therefore, each subcategory $\AV_{h,\Phi_i}(9)$ contains three isomorphism classes of abelian varieties; two with endomorphism ring $S$ and one with endomorphism ring $\cO_L$.
We compute that for each $\Phi_i$, only one of the abelian varieties with endomorphism ring $S$ is principally polarizable, with two non-isomorphic principal polarizations, while the abelian variety with maximal endomorphism ring admits a unique principal polarization (up to polarized isomorphisms).
\end{example}

The following example shows the need for a generalized residual reflex condition, by showing that it is possible for all the simple factors of a non-simple variety to satisfy the generalized residual reflex condition of Definition~\ref{def:RRC} while the variety itself does not, as claimed in Remark~\ref{rem:genRRConsimplefactors}. 

\begin{example}\label{ex:noCCOprod}
Consider the Weil $3$-polynomials 
\[ 
h_1=x^2+3 \qquad \text{ and } \qquad h_2=x^4+9
\]
and the isogeny class $\AV_{h_1h_2}(3)$ of abelian threefolds over $\F_3$,
see the LMFDB \cite{lmfdb} label \href{https://www.lmfdb.org/Variety/Abelian/Fq/3/3/a_d_a}{{3.3.a\_d\_a}}.
Using that both number fields $K_1=\Q[x]/(h_1)$ and $K_2=\Q[x]/(h_2)$ are Galois with a unique prime above $3$, 
one verifies that all CM-types of $\AV_{h_1}(3)$ and $\AV_{h_2}(3)$ satisfy the Shimura-Taniyama formula, cf.~Condition~\ref{def:RRC_item_st} of Definition~\ref{def:RRC}.
Moreover, since $K_1$ is a quadratic number field with a unique prime above $3$ with residue class field $\F_3$, both CM-types of  $\AV_{h_1}(3)$ also satisfy Condition~\ref{def:RRC_item_refl} of Definition~\ref{def:RRC}.
For $K_2$ one computes that only two out of the four possible CM-types satisfy Condition~\ref{def:RRC_item_refl} of Definition~\ref{def:RRC}.
In particular, by Theorem \ref{thm:CCOlift}, in both isogeny classes  $\AV_{h_1}(3)$ and $\AV_{h_2}(3)$ there exists an abelian variety that admits a canonical lifting to characteristic zero.
On the other hand, our algorithm shows that all CM-types of $\AV_{h_1h_2}(3)$ satisfy Condition~\ref{def:RRC_item_st} of Definition \ref{def:RRC} but no CM-type satisfies Condition \ref{def:RRC_item_refl} of Definition \ref{def:RRC}.
\end{example}

The following examples illustrate in detail how to apply the results from Section~\ref{ssec:polcomp}. In Example \ref{ex:detailedlowprank} we fully determine the number of principally polarized abelian varieties, 
showing that in fact we can ignore the contribution given by $\alpha$ for this isogeny class.

\begin{example}
\label{ex:detailedlowprank}
Consider the isogeny class $\AV_h(3)$ defined by $x^6 - x^5 + 6x^4 - 6x^3 + 18x^2 - 9x + 27$,
see the LMFDB \cite{lmfdb} label \href{https://www.lmfdb.org/Variety/Abelian/Fq/3/3/ab_g_ag}{{3.3.ab\_g\_ag}}.
It has $p$-rank~$1$ and it factors up to isogeny into a supersingular elliptic curve and an almost-ordinary surface.
Our computation shows that there exists a CM-type $\Phi$ of the endomorphism algebra $L=\Q[x]/h$ satisfying the generalized residue reflex condition, cf. Definition \ref{def:RRC}.

The order $R_w$ has index $18$ in the maximal order $\cO_L$. 
We found that for each divisor $d$ of~$18$ there exists an unique order $S_d$ of index $d$ in $\cO_L$. In particular, by Corollary \ref{cor:Snonempty}, we can deduce that $\mathcal{S}_\Phi$ contains $\cO_L$ and $S_2$. 
By applying Theorem~\ref{thm:eff}.\ref{thm:eff:2} we compute the principal polarizations of the isomorphism classes with endomorphism ring $S_2$ and $S_1=\cO_L$, and find that each is the endomorphism ring of two non-isomorphic abelian varieties admitting a unique principal polarization (up to polarized isomorphism).

Since Theorem \ref{thm:eff}.\ref{thm:eff:2} fails for the other overorders, we attempt to use Theorem \ref{thm:eff}.\ref{thm:eff:4}.
It turns out that the orbit $\Phi G_{S_2}$ consists of all eight CM-types of $L$.
Our computation returns that for every $\Phi'\in \Phi G_{S_2}$ and for every fractional ideal~$I$ with $(I:I)=S_6$ or $(I:I)=S_3$ the set $\Pone{\Phi'}{I}$ is empty.
We conclude that there is no principally polarized abelian variety with endomorphism ring $S_6$ or $S_3$.

Next, we consider $S_{18}$. Up to isomorphism, there are eight fractional ideals $I_1,\ldots, I_8$ with $(I_i:I_i)=S_{18}$.
We observe that for any $\Phi'\in \Phi G_{S_2}$ there are precisely four classes $I_i$ such that the set $\Pone{\Phi'}{I_i}$ is empty and four for which it contains one single element.
In other words, for any $\Phi'\in \Phi G_{S_2}$, the integer tuple
\begin{equation}
\label{eq:output_seq}
    \left( \vert \Pone{\Phi'}{I_1} \vert, \ldots, \vert \Pone{\Phi'}{I_8} \vert \right)
\end{equation}
is a permutation of 
\[ (1,1,1,1,0,0,0,0). \]

This means that we cannot apply Theorem \ref{thm:eff}.\ref{thm:eff:4} to compute the principal polarizations of a single $I_i$. 
But by Theorem \ref{thm:sameoutput} we can conclude that $4$ of the $8$ isomorphism classes with endomorphism ring $S_{18}=R_w$ admit a unique principal polarization (up to polarized isomorphism).
Applying the same reasoning to the order $S_9$ we can also deduce that $4$ of the $8$ isomorphism classes with endomorphism ring $S_{18}=R_w$ admit a unique principal polarization (up to polarized isomorphism).
\end{example}

Contrary to Example~\ref{ex:detailedlowprank}, in Example~\ref{ex:x8+16} below, we will see that we cannot ignore the contribution given by $\alpha$ for all endomorphism rings. Indeed, we find a finite set of possibilities for some endomorphism rings.

\begin{example}
\label{ex:x8+16}
Consider the isogeny class $\AV_h(2)$ defined by $h=x^8+16$,
see the LMFDB~\cite{lmfdb} label \href{https://www.lmfdb.org/Variety/Abelian/Fq/4/2/a_a_a_a}{{4.2.a\_a\_a\_a}}. 
The number field $L=\Q[x]/(h)=\Q(F)$ is Galois, and it is easy to verify that the residue field of $\cO_L$ at the unique prime above $2$ is $\F_2$. In particular we deduce that all CM-types satisfy the residual reflex condition.
In Table \ref{tab:x8_16} we summarize the output of the computation of the principal polarizations. 
The order $R_w=\Z[F,V]$ has $19$ overorders.
Each overorder $T$ is generated by $1$ and the conductor $(T:\cO_L)$, which is always a principal $\cO_L$ ideal.
Moreover, the index of each $T$ in $\cO_L$ is a power of $2$. 
Hence, we deduce that the set $\mathcal{S}$ contains only $\cO_L$. Notice that $15$ overorders $T$ satisfy $T = \overline{T}$.
Among those, only $3$ overorders satisfy Theorem \ref{thm:eff}.\ref{thm:eff:2}, and for $7$ others we can compute the number of isomorphism classes of principally polarized abelian varieties using Theorem \ref{thm:sameoutput}.
For the remaining $5$ overorders we have two possibilities. 
For example, there are $4$ (unpolarized) isomorphism classes of abelian varieties with endomorphism ring $R_w$: either all of them are not principally polarizable, or two of them admit $4$ non-isomorphic principal polarizations each while the other two are not principally polarizable.

This shows that $\alpha$ is a real obstruction in the following sense. We can construct a lifting using any of the CM-types (since they all satisfy the residual reflex condition). However, if we could put $\alpha=1$ for all these different CM-types, then we would have contradictory answers for the number of principally polarized abelian varieties in this isogeny class. 

\begin{table}[ht]
    \centering
    \small
    \bgroup
    \def\arraystretch{1.2}
    \begin{tabular}{|c|c|c|c|c|}\hline
order $T$ & $[\cO_L:T]$ & $T = \bar T$ & \ref{thm:eff}.\ref{thm:eff:2}$(T)$ & number of ppav \\ \hline
$ \cO_L $ & 1 & true & true & 1\\ \hline
$ \Z + \left(1+\frac 1 8 F^6\right)\cO_L $ & 2 & true & true & 1 \\ \hline
$ \Z + \frac 1 4F^5\cO_L $ & 4 & true & false & 0 \\\hline
$ \Z + \left(1+\frac 1 4F^4\right)\cO_L $& 4 & true & true &  1 \\\hline
$ \Z + \left(1-\frac 1 2F+\frac 1 4F^3- \frac 1 8F^5+\frac 1 8F^6+ \frac{1}{16}F^7\right)\cO_L $ & 4 & true & false & 0+0 \\ \hline
$ \Z + \left(\frac 1 2F^2+\frac 1 8F^6\right)\cO_L $ & 8 & true & false & 0+2+0+0 \\ \hline
$ \Z + \left(F-\frac 1 2F^3\right)\cO_L $  & 8 & false & -- & 0\\ \hline
$ \Z + \left(1+\frac 1 2F^2-\frac 1 4F^4-\frac 1 8F^6\right)\cO_L $  & 8 & false & -- & 0\\ \hline
$ \Z +  \frac 1 2F^4 \cO_L $  & 16 & true & false & 0\\ \hline
$ \Z + 2\cO_L $  & 16 & true & false & 2 or 0\\\hline
$ \Z + \left(F-\frac 1 2F^3\right)\cO_L $ & 16 & true & false & 0+0+0+0\\ \hline
$ \Z + \left(2+\frac 1 4F^6\right)\cO_L $ & 32 & true & false & 0 or 2\\\hline
$ \Z + \left(F-\frac 1 2F^3+\frac 1 4F^5-\frac 1 8F^7\right)\cO_L $ & 32 & true & false & 1+1 or 0+0\\ \hline
$ \Z +  \frac 1 2F^4 \cO_L $ & 32 & true & false & 0+2+2+0 or 0+0+0+0\\ \hline
$ \Z + \left(2-\frac 1 2F^4\right)\cO_L $ & 64 & false & -- & 0\\ \hline
$ \Z + \left(F-\frac 1 2F^3+\frac 1 4F^5+\frac 1 8F^7\right)\cO_L $ & 64 & true & false & 0+0+0+0\\ \hline
$ \Z + \frac 1 4F^7 \cO_L $ & 64 & false & -- & 0\\ \hline
$ \Z + F^3\cO_L $ & 128 & true & false & 0+0+0+0\\ \hline
$ R_w $ & 256 & true & false & 0+0+4+4 or 0+0+0+0\\ \hline
    \end{tabular}
    \egroup
    \caption{ For a conjugate stable order $T$, \ref{thm:eff}.\ref{thm:eff:2}$(T)$ is true if and only if  Theorem \ref{thm:eff}.\ref{thm:eff:2} holds for the order $T$. It is left blank if $T$ does not satisfy $T = \overline{T}$. The number of isomorphism classes of principally polarized abelian varieties with each endomorphism ring is given as an unordered sum where each summand is a distinct (unpolarized) isomorphism class. Note that in certain cases we have two possibilities.
    }
    \label{tab:x8_16}
\end{table}
\end{example}

\bibliographystyle{amsplain}
\renewcommand{\bibname}{References}
\bibliography{references}
\end{document}